\documentclass{amsart}

\newcommand{\bF}{\mathbb{F}}

\newcommand{\bN}{\mathbb{N}}

\newcommand{\bQ}{\mathbb{Q}}\newcommand{\bR}{\mathbb{R}}

\newcommand{\bZ}{\mathbb{Z}}

\usepackage[dvipsnames,svgnames,x11names,hyperref]{xcolor}
\usepackage{url,graphicx,verbatim,amssymb,enumerate,stmaryrd,nag}
\usepackage[pagebackref,colorlinks,citecolor=RoyalPurple,linkcolor=RoyalPurple,urlcolor=RoyalPurple,filecolor=RoyalPurple]{hyperref}
\usepackage{microtype}
\usepackage[all]{xy}
\usepackage[hmargin=3cm,vmargin=3cm]{geometry}

\usepackage{booktabs}

\newtheorem{theorem}{Theorem}[section]
\newtheorem{lemma}[theorem]{Lemma}
\newtheorem{proposition}[theorem]{Proposition}
\newtheorem{corollary}[theorem]{Corollary}

\theoremstyle{definition}
\newtheorem{definition}[theorem]{Definition}
\newtheorem{example}[theorem]{Example}
\newtheorem{remark}[theorem]{Remark}

\newtheorem{convention}[theorem]{Convention}

\newtheorem{digression}[theorem]{Discussion}

\newcommand{\Z}{\mathbb{Z}}
\newcommand{\Q}{\mathbb{Q}}
\newcommand{\F}{\mathbb{F}}
\newcommand{\R}{\mathbb{R}}

\newcommand{\FI}{\cat{FI}}
\newcommand{\cat}[1]{\mathsf{#1}}
\newcommand{\mr}[1]{{\rm #1}}
\newcommand{\fS}{\mathfrak{S}}
\newcommand{\Homz}[1]{\mr{Hom}_{\Z}(#1,\Z)}
\newcommand{\Homq}[1]{\mr{Hom}_{\Z}(#1,\Q)}
\newcommand{\Extz}[1]{\mr{Ext}^1_{\Z}(#1,\Z)}
\newcommand{\Diff}{\mr{Diff}}
\newcommand{\Homeo}{\mr{Homeo}}

\newcommand{\I}{\cat{I}}

\title{Representation stability for homotopy groups of configuration spaces}
\author{Alexander Kupers}
\thanks{Alexander Kupers is supported by a William R. Hewlett Stanford Graduate Fellowship, Department of Mathematics, Stanford University, and was partially supported by NSF grant DMS-1105058.}
\author{Jeremy Miller}
\date{\today}

\begin{document}

\begin{abstract}We prove that the dual rational homotopy groups of the configuration spaces of a $1$-connected manifold of dimension at least 3 are uniformly representation stable in the sense of \cite{CF}, and that their derived dual integral homotopy groups are finitely generated as $\cat{FI}$-modules in the sense of \cite{CEF}. This is a consequence of a more general theorem relating properties of the cohomology groups of a $1$-connected co-$\FI$-space to properties of its dual homotopy groups. We also discuss several other applications, including free Lie and Gerstenhaber algebras.\end{abstract} %% Just for papers exceeding 50 pages.

\maketitle

\tableofcontents

\section{Introduction}Let $F_k(M)$ denote the ordered configuration space of $k$ distinct points in a manifold $M$ and $C_k(M)$ the unordered configuration space. That is, we have 
\[F_k(M) = \{(m_1,m_2, \ldots, m_k) \,|\, m_i \neq m_j \text{ for } i\neq j\} \subset M^k \quad \text{ and } \quad  C_k(M)=F_k(M)/\fS_k\]
where in the second case the symmetric group on $k$ letters $\fS_k$ permutes the terms of the product. The goal of this paper is to prove a stability theorem for $\pi_i(C_k(M))$ for $i \geq 2$ in the case that $M$ is a $1$-connected (connected and simply connected) manifold of dimension at least $3$. 

\begin{convention}We will assume that our manifolds are of finite type. This is in particular the case for interiors of manifolds that admit a finite handle decomposition.\end{convention}

\subsection{Homological stability and lack of homotopical stability for configuration spaces}

One says that a sequence of spaces $X_k$ exhibits homological stability if the isomorphism types of the homology groups $H_i(X_k)$ are independent of $k$ for $k \gg i$. In \cite{Mc1}, McDuff showed that unordered configuration spaces of points in non-compact manifolds exhibit homological stability. In \cite{Ch}, this was generalized by Church to include closed  orientable manifolds if one works with rational coefficients. This result is especially surprising given the fact that there are no natural maps between unordered configuration spaces of points in a closed manifold. 

Do the homotopy groups also stabilize? Interpreted literally, the answer is no. For example, $\pi_1(C_k(\bR^2))$ is the braid group on $k$ strands and when $M$ is of dimension at least $3$, $\pi_1(C_k(M))$ is the wreath product $\pi_1(M) \wr \fS_k$. For higher homotopy groups, it follows from the work of Cohen in III.4 of \cite{CLM} that $\pi_{n-1}(C_k(\bR^n))  \cong \bZ^{k \choose 2}$ for $n>2$, so these similarly do not stabilize.

\subsection{Representation stability for rational homotopy groups of configuration spaces}

Note that although the fundamental groups do not stabilize, in dimensions at least $3$ they admit a uniform description in terms of symmetric groups. We will show that when $M$ is $1$-connected, not only does the fundamental group admit a uniform description, but so do the higher homotopy groups eventually. This description will be in terms of representations of symmetric groups. Indeed, if $M$ is $1$-connected and of dimension at least $3$, then the action of the fundamental group on the higher homotopy groups makes $\pi_n(C_k(M)) \otimes \bQ$ into a rational $\fS_k$-representation since now $\pi_1(C_k(M)) \cong \fS_k$.

In \cite{CF} Church and Farb defined representation stability for a sequence $V_k$ of rational $\fS_k$-representations. First one gives names to irreducible $\fS_k$-representations that do not depend on $k$. For example, the trivial and standard representations make sense for all $\fS_k$. This allows one to compare the multiplicities of $\fS_k$-representations for different $k$. One then says that the sequence $V_k$ has representation stability if the multiplicities of the irreducible subrepresentations of $V_k$ are eventually constant, see Definition \ref{defrepstab}.

Since the multiplicities can stabilize without the dimensions stabilizing, representation stability gives a notion of stability where traditional homological stability fails. A striking example of representation stability is Church's result that the cohomology groups of the ordered configuration space $F_k(M)$ have representation stability \cite{Ch}. Since then, representation stability has been proven for the cohomology of many other sequences of spaces \cite{Wi} \cite{JR} \cite{CEF}. One of the main results of this paper is the first result concerning representation stability for homotopy groups:

\begin{theorem}
\label{config}

Let $M$ be a $1$-connected manifold of dimension at least $3$. For $i \geq 2$, the dual rational homotopy groups $\Homq{\pi_i(C_k(M))}$ are representation stable with range $k \geq 4(i-1)$. If we additionally assume that $M$ is non-compact, this can be improved to $k \geq 2(i-1)$.

\end{theorem}

Note that rational $\fS_k$-representations are self-dual in the sense that they are non-canonically isomorphic to their duals. Thus the theorem implies that the multiplicities of the irreducible representations in $\pi_i(C_k(M)) \otimes \bQ$ also stabilize. 

\subsection{$\FI$-modules and integral results}

It is natural to ask if a similar statement holds for the integral homotopy groups. Because the theory of $\bZ[\fS_k]$-modules is much more complicated than the theory of $\bQ[\fS_k]$-modules, there is no obvious analogue of the definition of representation stability for sequences of abelian groups with symmetric group actions. However, in \cite{CEF}, Church, Ellenberg and Farb introduced the machinery of $\FI$-modules to streamline proofs of representation stability results and gave a condition that does not involve naming irreducible representations, but implies representation stability when working over $\bQ$.

Let $\cat{FI}$ denote the category of finite sets and injective maps. For a ring $R$, an $\FI$-$R$-module is a functor from $\cat{FI}$ to the category of $R$-modules. Finite generation of an $\cat{FI}$-$R$-module means that a finite subset of its elements generates the entire $\cat{FI}$-$R$-module under the structure maps. The value of an $\FI$-$R$-module on $\{1,\ldots,k\}$ naturally comes with an action by $\fS_k$, because these are the automorphisms of the set $\{1,\ldots,k\}$. In this way, the data of an $\FI$-$\bQ$-module includes the data of a sequence of $\fS_k$-representations. If the $\FI$-$\Q$-module is finitely generated, Church, Ellenberg and Farb show that the sequence of representations exhibits  representation stability \cite{CEF}. 

Let us return to configuration spaces. Forgetting the ordering gives a $k!$-sheeted covering map $F_k(M) \to C_k(M)$ and hence $\pi_i(F_k(M)) \cong \pi_i(C_k(M))$ for $i \geq 2$. Moreover, when $M$ is a $1$-connected manifold of dimension at least $3$, the action of the fundamental group on the higher homotopy groups of $C_k(M)$ agrees with the $\fS_k$-action on $\pi_i(F_k(M))$ induced by the action of $\fS_k$ on the space $F_k(M)$ by permuting the ordering of the points. Therefore, Theorem \ref{config} can be rephrased as representation stability for the dual rational homotopy groups of $F_k(M)$. The spaces $F_k(M)$ naturally have the structure of a contravariant functor from $\FI$ to spaces and so their integral cohomology groups are $\FI$-$\bZ$-modules. %In fact, the definition of $\FI$-$\bZ$-module was designed to axiomatize the ``forget the $i$th point'' maps from $F_{k+1}(M)$ to $F_k(M)$. 

Because the theory of $\FI$-$\bZ$-modules is better understood than the theory of co-$\FI$-$\bZ$-modules, we follow the rational case by considering dual homotopy groups instead of homotopy groups. In the integral setting we should take derived duals, that is, consider both $\Homz{-}$ and $\Extz{-}$. This has the effect of separating the free and torsion parts. Due to base point issues, the dual homotopy groups of a co-$\FI$-space are not naturally an $\FI$-$\bZ$-module. However, when the co-$\FI$-space is $1$-connected (the value on each finite set is $1$-connected), the dual homotopy groups can uniquely be given the structure of $\FI$-$\bZ$-modules. The second main theorem of this paper is the following integral analogue of Theorem \ref{config}:

\begin{theorem}
Let $M$ be a $1$-connected manifold of dimension at least $3$. For $i \geq 2$, the $\FI$-$\bZ$-modules $\Homz{\pi_i(C_k(M))}$ and $\Extz{\pi_i(C_k(M))}$ are finitely generated. \label{configint}
\end{theorem}

If $M$ is non-compact, one can use Proposition \ref{propfishdualz} to deduce a similar statement about $\pi_i(C_k(M))$. See Remark \ref{remdim2} for the situation in dimension $\leq 2$. From this finite generation result, we can among other things deduce the following, using Propositions \ref{proptheoremc} and \ref{propfgtorsann}.

\begin{corollary} \label{intconfigcor} Let $M$ be a $1$-connected manifold of dimension at least $3$ and let $i \geq 2$. \begin{enumerate}[(i)]
\item There is a natural number $N_i$ such that multiplication by $N_i$ annihilates all of the torsion in $\pi_i(C_k(M))$ for all $k$.
\item For all fields $\bF$ there exists a polynomial $P_i$ such that the function $\dim_\bF(\pi_i(C_k(M)) \otimes \bF)$ is equal to $P_i(k)$ if $k$ is sufficiently large 
\end{enumerate}
\end{corollary}

\subsection{Other fundamental groups} 
It is natural to ask if these results also apply to the case of manifolds that are connected but not simply connected. In that case the fundamental groups of the unordered configuration spaces are no longer isomorphic to symmetric groups. In dimension at least $3$, the fundamental group is instead the wreath products of the symmetric group with the fundamental group of the manifold. These can be treated in a similar framework to $\FI$-modules by replacing $\FI$ with $\FI(G)$, a category introduced by Sam and Snowden in \cite{SS2}. In Section \ref{secG}, we prove a generalization of Theorem \ref{configint} addressing the case that the manifold has polycyclic-by-finite fundamental group. It is interesting to ask whether Theorem \ref{configint} can be generalized to all connected manifolds of dimension of at least $3$ and if so, what form this generalization takes.

\subsection{Relationship between finite generation for cohomology and dual homotopy groups of $\FI$-modules}

In \cite{CEFN}, Church, Ellenberg, Farb and Nagpal proved that the cohomology groups of ordered configuration spaces form finitely generated $\FI$-modules. We will deduce our results about homotopy groups from their results about cohomology groups by proving that a $1$-connected finitely generated co-$\FI$-space has finitely generated cohomology if and only if it has finitely generated dual (in the derived sense) homotopy groups. This is done using the main technical theorem of this paper, proven in Section \ref{secmain} (Theorems \ref{thmequivint} and \ref{thmequivrat} to be precise). The second half of it uses stability degree and weight, whose definition we will recall in Section \ref{secFI}. They should be thought of as different ways of making finite generation quantitative when working over the rationals.

\begin{theorem}Let $X$ be a $1$-connected co-$\FI$-space of finite type.
\label{thmequiv}
\begin{enumerate}[(i)]

\item If $H^i(X;\bZ)$ is a finitely generated $\FI$-$\bZ$-module for all $i$, then so are $\Homz{\pi_i(X)}$ and $\Extz{\pi_i(X)}$  for all $i$.
\item Conversely, if $\Homz{\pi_i(X)}$ and $\Extz{\pi_i(X)}$  are finitely generated $\FI$-$\bZ$-modules for all $i$, then so is $H^i(X;\bZ)$ for all $i$.

\end{enumerate}

It is possible to give explicit ranges if we work rationally. Fix $c \in \bQ$.
\begin{enumerate}[(i)]\setcounter{enumi}{2}
\item Suppose that $H^i(X;\bQ)$ has weight $\leq ci$ and stability degree $\leq ci$, then $\mr{Hom}_\bZ(\pi_i(X),\bQ)$ has weight $\leq 2c(i-1)$ and stability degree $\leq 4c(i-1)$.
\item Suppose $\mr{Hom}_\bZ(\pi_i(X),\bQ)$ has weight $\leq ci$ and stability degree $\leq ci$, then $H^i(X;\bQ)$ has weight $\leq ci$ and stability degree $\leq c(2i+1)$.
\end{enumerate}

\end{theorem}

The above theorem can be applied to examples other than configuration spaces. In particular, by considering a wedge of spheres, we will be able to deduce representation stability for free Lie algebras and free Gerstenhaber algebras (see Section \ref{algebra}).

\subsection{Organization of the paper}

In Section \ref{secFI}, we review the relevant properties of $\FI$-modules established in \cite{CEF} and \cite{CEFN}, and prove a few results missing from the literature. In Section \ref{secbasepoints}, we discuss technical issues involving base points. In Section \ref{secmain}, we prove Theorem \ref{thmequiv}.  In Section \ref{secap}, we discuss applications, such as Theorems \ref{config} and \ref{configint}, as well as stability results for homotopy groups of certain subgroups of automorphisms groups of manifolds, and for free Lie and Gerstenhaber algebras. 

\subsection{Acknowledgements} We would like to thank Martin Bendersky, Thomas Church, S{\o}ren Galatius, Steven Sam and Jennifer Wilson for several helpful conversations.

\section{$\FI$-modules and their properties}
\label{secFI}

To prove Theorem \ref{thmequiv}, we will need to show that various constructions in algebraic topology preserve finite generation of $\FI$-$R$-modules. This will require understanding how finite generation and related concepts interact with taking tensor products, images, kernels and extensions. Much of this section is a summary of results of \cite{CEF} and \cite{CEFN}, and we recommend those for more details and exposition.

\subsection{$\FI$-$\bZ$-modules and finite generation} We start by recalling the definition of $\FI$-modules and finite generation. After that, we describe $\FI\#$-modules, which were also introduced in \cite{CEF} and capture the extra structure that the cohomology of configuration spaces of ordered points in open manifolds possess. We also look at the implications of finite generation and discuss constructions that preserve finite generation.

\begin{definition}Let $\FI$ be the category of finite sets and injections. For a ring $R$, an \emph{$\FI$-$R$-module} is a functor from $\FI$ to the category of $R$-modules.\end{definition}

For $F$ an $\FI$-$R$-module, let $F_k$ denote the value of $F$ on $[k] := \{1,\ldots k\}$.

\begin{definition}
An $\FI$-$R$-module $F$ is \emph{finitely generated} if there is a finite set $S$ of elements in $\bigoplus_{k=0}^\infty F_k$ so that no proper sub-$\FI$-$R$-module of $F$ contains $S$.
\end{definition}

\begin{example}

Let $F$ be the $\FI$-$R$-module whose value on a set $S$ is the free $R$-module on that set. Let $a \in F([1])$ be non-zero and $b \in F([2])$ be an element not fixed by the $\fS_2$ action. One can check that these are $\FI$-module generators of $F$. 
\end{example}

\subsubsection{Tensor products, quotients, submodules and extensions}

Finite generation is well-behaved with respect to tensor products, quotients, submodules and extensions, as summarized in the following table. 

\begin{table}[h]
\begin{center}
\begin{tabular}{@{}llllllll@{}} \toprule
    {} & \phantom{abc}  & {finite generation}  \\ \midrule
    tensor products  &   & preserved  \\
    quotients  &   & preserved   \\
    submodules & & preserved if $R$ is Noetherian \\
    extensions  &   & preserved   \\ \bottomrule
\end{tabular}\end{center}
%\vspace{\baselineskip}
%\caption{The behavior of finite generation under various operations.}
\label{tabfg}
\end{table}

In more detail, we have the following proposition, which combines Proposition 2.3.6 of \cite{CEF}, Theorem A of \cite{CEFN} and two easy arguments implicit in \cite{CEF}.

\begin{proposition}\label{propfgprops}\begin{enumerate}[(i)]
\item If $F$ and $G$ are finitely generated $\FI$-$R$-modules, then so is $F \otimes_R G$.
\item If $G$ is a finitely generated $\FI$-$R$-module, $F$ is an arbitrary $\FI$-$R$-module and $F \hookrightarrow G$ is a injective map of $\FI$-$R$-modules, then the quotient $\FI$-$R$-module $G/F$ is finitely generated.
\item Suppose $R$ is a Noetherian ring, $F$ a sub-$\FI$-$R$-module of an $\FI$-$R$-module $G$. If $G$ is finitely generated, so is $F$.
\item If an $\FI$-$R$-module $F$ has a finite filtration $0 \subset F_1 \subset \ldots \subset F_k = F$ by $\FI$-$R$-modules such that $F_i/F_{i-1}$ is finitely generated, then $F$ is finitely generated.
\end{enumerate}\end{proposition}

A direct consequence of this proposition is that finite generation is preserved by spectral sequences. More precisely, we have the following corollary.

\begin{corollary}\label{corssfg} Suppose that $\{E^{p,q}_r(k)\}_{k \geq 0}$ is a spectral sequence of $\FI$-$R$-modules converging to an $\FI$-$R$-module $F^{p+q}(k)$ such that each filtration on $E_\infty$ has finitely many non-zero terms. Suppose furthermore that $R$ is Noetherian. If there is an $r \geq 1$ so that $E^{p,q}_r$ is finitely generated for all $p,q$, then $F^{p+q}(k)$ is finitely generated. A similar result holds in a range.\end{corollary}

\subsubsection{$\FI\#$-modules and the functor $H_0$}

In \cite{CEF}, Church, Ellenberg and Farb also introduced a more restrictive notion than $\FI$-modules, $\FI\#$-modules. This structure axiomatizes the interaction between McDuff's maps bringing in a particle from infinity \cite{Mc1}, and the  maps that forget the $i$th point. 

\begin{definition}

Let $\FI\#$ denote the category whose objects are finite sets and whose morphisms are defined as follows. For $S$ and $T$ finite sets, let $\mr{Hom}_{\FI\#}(S,T)$ be the set of triples $(A,B,\phi)$ with $A \subset S$, $B \subset T$ and $\phi: A \to B$ a bijection. Composition of morphisms is given by composition of functions where the domain and codomain are the largest possible making the composition a well defined bijection. For a ring $R$, an \emph{$\FI\#$-$R$-module} is a functor from $\FI\#$ to the category of $R$-modules.
\end{definition}

Note that $\FI\#$-$R$-modules are naturally $\FI$-$R$-modules, via the inclusion $\FI \hookrightarrow \FI\#$ given by the identity on objects and sending an injection $\sigma: S \to T$ to $(S,\sigma(S),\sigma)$.

\begin{definition}
For an $\FI$-$R$-module $F$, let $H_0(F)_k$ be the $R[\fS_k]$-module $F_k/\mr{span}(F_{<k})_k$. Here $\mr{span}(F_{<k})$ is the intersection of all $\FI$-$R$-submodules of $F$ containing all $F_i$ for $i <k$.  
\end{definition} 

Note that $H_0(F)$ is finitely generated if and only if $F$ was. The following construction from \cite{CEF} will be used in the proof of Proposition \ref{proptensorrat} and plays a significant role in the classification of $\FI\#$-$R$-modules.

\begin{definition}
For $V$ a $R[\fS_n]$-module, let $M(V)$ be the $\FI\#$-$R$-module given as follows: $M(V)_k = R[\mr{Hom}_\FI([n],[k])] \otimes_{R[\fS_n]} V$.\end{definition}

If $\{V_n\}_{n \geq 0}$ is a sequence of $R[\fS_n]$-modules, we define $M(\{V_n\}_{n \geq 0}) = \bigoplus_n M(V_n)$. The classification of $\FI\#$-$R$-modules in Theorem 4.1.5 of \cite{CEF} says every $\FI\#$-$R$-module is of this form, by showing that there is a natural isomorphism $M(H_0(F)) \cong F$. The classification implies Theorem 4.1.7 of \cite{CEF}, part of which says that an $\FI\#$-$\bZ$-module $F$ is finitely generated if and only if $F_n$ is generated as an abelian group by $O(n^d)$ elements.

The category $\FI\#$ is self-dual, i.e. there is an equivalence of categories $\eta: \FI\# \to \FI\#^\mr{op}$ given by sending an object $S$ to itself and a morphism $(A,B,\phi)$ to $(B,A,\phi^{-1})$. This can be used in combination with Theorem 4.1.7 of \cite{CEF} to prove the following duality property of finitely generated $\FI\#$-$\Z$-modules.

\begin{proposition}\label{propfishdualz} An $\FI\#$-$\Z$-module $F$ is finitely generated if and only if  $\Homz{F} \circ\eta$ and $\Extz{F} \circ \eta$ are finitely generated.\end{proposition}

\begin{proof}It suffices to remark that $\Homz{F}_k$ has number of generators equal to the rank of $F_k$ and $\Extz{F}_k$ has number of generators equal to the number of generators of the torsion part of $F_k$.\end{proof}

\subsubsection{Consequences of finite generation over $\Z$}

In this section we describe some concrete consequences of finite generation for $\FI$-$\Z$-modules. There are more, but we find the following two esthetically pleasing. We start with Theorem B of \cite{CEFN}. 

\begin{proposition}\label{proptheoremc} Let $F$ be a finitely generated $\FI$-$\Z$-module and $\bF$ a field, then for all except finitely many values of $k$, $\dim F_k \otimes \bF$ is equal to the value of a polynomial in $k$. If $F$ is $\FI\#$, this is true for all values of $k$.
 \end{proposition}

The following is a consequence of Theorem C of \cite{CEFN}. 
 
\begin{proposition}\label{propfgtorsann} Let $F$ be a finitely generated $\FI$-$\Z$-module. There exists an integer $N \geq 1$ such that $N$ annihilates all of the torsion of $F_k$ for all $k \geq 0$.
\end{proposition} 
 
\begin{proof}Let $F[\mr{tors}]_k$ be the torsion subgroup of $F_k$. Since maps of abelian groups must send torsion elements to torsion elements $\{F[\mr{tors}]_k\}_{k \geq 0}$ is a sub-$\FI$-$\Z$-module of $F$. Because $F$ is finitely generated, Part (iii) of Proposition \ref{propfgprops} implies that $F[\mr{tors}]$ is finitely generated. Theorem C of \cite{CEFN} then implies that there exists an integer $M$ such that $F[\mr{tors}]_k$ can be written as a colimit of a diagram indexed by subsets of cardinality $\leq M$ of $\{1,\ldots,k\}$ of abelian groups isomorphic to $F[\mr{tors}]_i$ for $i \leq M$. Suppose that $N$ annihilates all $F[\mr{tors}]_i$ for $i \leq M$, then $N$ annihilates this colimit.
\end{proof}

\subsection{$\FI$-$\bQ$-modules, stability degree, weight and generation degree} In contrast with the integral case, working over $\bQ$ allows us to make quantitative statements. This is because, unlike $\bZ[\fS_k]$-modules, finite-dimensional $\bQ[\fS_k]$-modules can be completely classified. Quantitative statements can made in terms of one of three related concepts; weight (having to do with the complexity of the representations), stability degree (having to do with the multiplicity of the representations) and degree of generation (being a mix of both of these concepts).

The definition of stability degree will involve the functors $\Phi_q$, whose definition involves the coinvariants $V_G$, defined as the quotient of $V$ by the subspace generated by $v-g\cdot v$.

\begin{definition} For $F$ an $\FI$-$\bQ$-module and $k \geq 0$, let $\Phi_q(F)_k=(F_{k+q})_{\fS_k}$ and let $\Phi_q(F)=\bigoplus_{k \geq 0} \Phi_q(F)_k$. Let $T: \Phi_q(F)_k \to  \Phi_q(F)_{k+1}$ be the stabilization map induced by the standard inclusion of $[q+k]$ into $[q+k+1]$.
\end{definition}

To define weight, we recall that there is a bijective correspondence between irreducible rational $\fS_k$-representations and partitions of $k$. A partition of $k$ is a collection of integers $\lambda= (\lambda_1 \geq \ldots \geq \lambda_l >0)$ satisfying $\sum_{i=1}^l \lambda_i = k$. For $\lambda$ a partition of $k$, let $|\lambda|$ denote $k$.  These partitions can be visualized by Young tableaux. Our notation follows \cite{CF}:

\begin{definition}\label{defirrepnot} Let $\lambda$ be a partition of $k$. \begin{enumerate}[(i)]
\item  Let $V_\lambda$ denote the irreducible representation of $\fS_k$ corresponding to the partition $\lambda = (\lambda_1 \geq \ldots \geq \lambda_l > 0)$ of $k$.
\item For $n \geq k + \lambda_1$, let $V(\lambda)_n$ denote the irreducible representation of $\fS_n$ corresponding to the partition $(n-k \geq \lambda_1 \geq \ldots \geq \lambda_l > 0)$ of $n$.
\end{enumerate}
\end{definition}

The families of representations $V(\lambda)_n$ are often familiar. For example, $V(0)_n$ is the trivial representation and $V(1)_n$ is the standard representation, the $(n-1)$-dimensional subrepresentation of the permutation representation obtained by removing the trivial subrepresentation. We will now give the definition of weight, stability degree and degree of generation.

\begin{definition}Let $F$ be an $\FI$-$\Q$-module.
\begin{enumerate}[(i)]
\item We say $F$ has \emph{stability degree $\leq r$} if for all $q \geq 0$ and $k \geq r$ the map $T:\Phi_q(F)_k \to \Phi_q(F)_{k+1}$ is an isomorphism.
\item We say $F$ has \emph{weight $\leq s$} if for every $n \geq 0$ and every irreducible constituent $V(\lambda)_n$ of $F_n$, we have $|\lambda| \leq s$.
\item We say $F$ is \emph{finitely generated in degree $\leq d$} if there is a finite set $S$ of elements in $\bigoplus_{n=0}^d F_n$ so that no proper sub-$\FI$-$\Q$-module of $F$ contains $S$.
\end{enumerate} 
\end{definition}

The next table compiles results that have been proven in \cite{CEF}. Our goal in the remainder of this section is to recall precise statements for each of the entries and conditionally fill in the remaining gaps in Lemma \ref{lemkerimgendeg} and Proposition \ref{proptensorrat}.

\begin{table}[h]
\begin{center}
\begin{tabular}{@{}llllllll@{}} \toprule
    {} & \phantom{abc}  & {stability degree} & {weight} & {generated below degree} \\ \midrule
    tensor products  & &  & additive & additive \\
    $\ker/\mr{im}$  &   & preserved  & preserved &  \\
    extensions  &   & preserved  & preserved & preserved &  \\ \bottomrule
\end{tabular}\end{center}
%\vspace{\baselineskip}
%\caption{The behavior of stability degree, weight and degree of generation under various operations.}
\label{tabrat}
\end{table}

\subsubsection{Representation stability} To fill in the gaps in the previous table and describe a concrete consequence of finite generation, we recall the definition of representation stability. 

A collection $\{V_n\}_{n \in \bN}$ of rational $\fS_n$-representations and maps $\phi_n: V_n \to V_{n+1}$ is said to be \emph{consistent} if each $\phi_n$ is a $\fS_n$-equivariant map $V_n \to V_{n+1}$, viewing $V_{n+1}$ as a $\fS_{n}$-representation via the standard inclusion $\fS_n \to \fS_{n+1}$. An $\FI$-$\Q$-module $F$ naturally gives rise to a consistent sequence by setting $V_n = F([n])$ and $\phi_n: V_n \to V_{n+1}$ equal to the map induced by the standard injection $[n] \hookrightarrow [n+1]$.

\begin{definition}
\label{defrepstab}
A consistent sequence of rational $\fS_n$-representations $V_n$ is said to be \emph{uniformly representation stable} with range $n \geq N$ if it has the following properties for $n \geq N$:
\begin{enumerate}[(i)]
\item the maps $\phi_n: V_n \to V_{n+1}$ are injective,
\item the image $\phi_n(V_n)$ spans $V_{n+1}$ as a $\fS_{n+1}$-module,
\item for each partition $\lambda$ the multiplicity of $V(\lambda)_n$ in $V_n$ is independent of $n$.
\end{enumerate}

%It is said to be monotone with range $n \geq N$ if we have
%\begin{enumerate}[(i)] \setcounter{enumi}{3}
%\item Monotonicity: for any $n \geq N$ and each subrepresentation $W \subset V_n$ isomorphic to $V(\lambda)_n^{\oplus k}$ the $\fS_{n+1}$-span of $W$ contains $V(\lambda)_{n+1}^{\oplus k}$ as a subrepresentation.
%\end{enumerate}

\end{definition}

Here is the key lemma describing the interplay of stability degree, weight and degree of generation.

\begin{lemma}
\label{FIrepstab}

Let $F$ be an $\FI$-$\Q$-module. If $F$ has stability degree $\leq r$ and weight $\leq s$, then it is  uniformly representation stable with range $\geq r+s$ and generated in degree $\leq r+s$.
\end{lemma}

\begin{proof}The first half is Proposition 3.3.3 of \cite{CEF} and the second half follows by noting that property (ii) of being uniformly representation stable implies generation in degree $\leq r+s$.\end{proof}

\subsubsection{$\FI\#$-$\Q$-modules} $\FI\#$-$\bQ$-modules are useful to consider, because for them all the three different quantitative notions of stability collapse to weight.

\begin{proposition}
\label{propdegweight}
If $F$ is an $\FI$-$\bQ$-module generated in degrees $\leq d$ then $F$ has weight $\leq d$. If $F$ also has the structure of an $\FI\#$-$\bQ$-module, then the converse is also true. 
\end{proposition}

Since rational $\fS_k$-representations are self-dual, sending an $\FI\#$-$\bQ$-module $F$ to $\mr{Hom}_\bQ(F,\bQ) \circ \eta$ preserves weight:

\begin{proposition}\label{propweightdualrat} An $\FI\#$-$\Q$-module $F$ has weight $\leq d$ if and only if $\mr{Hom}_\bQ(F,\bQ) \circ \eta$ has weight $\leq d$.\end{proposition}

\subsubsection{Kernels, images and extensions}

In this section, we discuss the behavior of finite generation, weight, and stability degree under kernel mod image and extensions. The following combines Lemma 3.1.6 of \cite{CEF} and a direct consequence of the definition of weight.

\begin{proposition}\label{propkerimrat}
\begin{enumerate}[(i)]
\item If $F \overset{f}{\to} G \overset{g}{\to} H$ is a sequence of $\FI$-$\Q$-modules with $g \circ f = 0$ and $F$, $G$ and $H$ have stability degree $\leq r$, then $\ker(g)/\mr{im}(f)$ has stability degree $\leq r$.
\item If $F \overset{f}{\to} G \overset{g}{\to} H$ is a sequence of $\FI$-$\Q$-modules with $g \circ f =0$ and $G$ has weight $\leq s$, then $\ker(g)/\mr{im}(f)$ has weight $\leq s$.
\end{enumerate}
\end{proposition}

We remark that weight is preserved by taking submodules and quotients. This is not true for stability degree or degree of generation, but we can still deduce the following:

\begin{lemma}\label{lemkerimgendeg} If $F \overset{f}{\to} G \overset{g}{\to} H$ is a sequence of $\FI$-$\Q$-modules with $g \circ f = 0$ and $F$, $G$ and $H$ have stability degree $\leq r$ and are generated in degree $\leq d$, then $\ker(g)/\mr{im}(f)$ is generated in degree $\leq d+r$.\end{lemma}

\begin{proof}By Proposition \ref{propdegweight}, we have that $F,G,H$ have weight $\leq d$. Since taking $\ker(g)/\mr{im}(f)$ preserves the stability degree and weight by Proposition \ref{propkerimrat}, it has stability degree $\leq r$ and weight $\leq d$. By Lemma \ref{FIrepstab} we have that $\ker(g)/\mr{im}(f)$ is generated in degree $\leq d+r$.\end{proof}

Next we consider extensions. The following proposition is implicit in \cite{CEF} and easy to prove. In particular, note that Part (i) follows from exactness of $\Phi_a$ over $\bQ$.

\begin{proposition}Suppose an $\FI$-$\Q$-module $F$ has a finite filtration $0 \subset F_1 \subset \ldots \subset F_k = F$ by $\FI$-$\Q$-modules. \label{propextrat}
\begin{enumerate}[(i)]
\item If each filtration quotient $F_i/F_{i-1}$ has stability degree $\leq r$, then so has $F$.
\item If each filtration quotient $F_i/F_{i-1}$ has weight $\leq s$, then so has $F$.
\item If each filtration quotient $F_i/F_{i-1}$ is generated in degree $\leq d$, then so is $F$.
\end{enumerate}
\end{proposition}

These results imply that one can use spectral sequences to prove bounds on weight or stability degree. In Proposition \ref{proplooprat} and Theorem \ref{thmequivrat}, we will apply this observation to the Eilenberg-Moore spectral sequence and the Quillen spectral sequence respectively.

\subsubsection{Tensor products} We next discuss tensor products of $\FI$-$\bQ$-modules. Two of the cases are dealt with by Propositions 2.3.6 and 3.2.2 of \cite{CEF}, which say:

\begin{proposition}\label{proptensorgeneweight} Let $F,G$ be $\FI$-$\Q$-modules.
\begin{enumerate}[(i)]
\item If $F$ and $G$ are generated in degrees $\leq d_1,d_2$ respectively, then $F \otimes G$ is generated in degree $\leq d_1+d_2$.
\item If $F$ and $G$ have weight $\leq s_1,s_2$ respectively, then $F \otimes G$ has weight $\leq s_1+s_2$.
\end{enumerate}
\end{proposition}

Next, we explain how stability degree ranges behaves under tensor products. \begin{comment}This is not obvious from the previous proposition, as there is a no direct way to go from weight and degree of generation to stability degree.We will need a lemma's. The first is combines Propositions 3.2.4 and 3.1.7 of \cite{CEF} and a direct consequence of the definition of weight:

\begin{lemma}
\label{lemfreeweightstabdeg}
Let $\lambda= (\lambda_1 \geq \ldots \geq \lambda_l >0)$ be a partition. The $\FI$-$\Q$-module $M(\lambda)$ has stability degree $\leq \lambda_1$ and weight $\leq |\lambda|$.
\end{lemma}

The second one is Proposition 3.2.8 of \cite{CEF} and shows us that stability degree tells us something about which summands $V(\lambda)_n$ can appear.

\begin{lemma}\label{lemstabdegvln} Let $F$ be an $\FI$-$\bQ$-module with stability degree $\leq r$, then in each $V(\lambda)_n$ we have $\lambda_1 \leq r$, where $\lambda_1$ is the largest number in the partition $\lambda$.\end{lemma}

\end{comment}

\begin{proposition}\label{proptensorrat} If $\FI$-$\bQ$-modules $F$ and $G$ have stability degree $\leq r_1,r_2$ and weight $\leq s_1,s_2$ respectively, then $F \otimes G$ is generated in degree $\leq r_1+r_2+s_1+s_2$, has weight $\leq s_1+s_2$ and stability degree $\leq \max(r_1+s_1,r_2+s_2,s_1+s_2) \leq r_1+r_2+s_1+s_2$.

More generally, given a collection $F_1,\ldots,F_k$ of $\FI$-$\bQ$-modules having stability degrees $\leq r_1,\ldots,r_k$ and weight $\leq s_1,\dots,s_k$ respectively, $F_1 \otimes \ldots \otimes F_k$ is generated in degree $\leq r_1+\ldots+r_k+s_1+\ldots+s_k$, has weight $\leq s_1+\ldots+s_k$ and stability degree $\leq$ $ \max(r_1+s_1,\ldots,r_k+s_k,s_1+\ldots+s_k)$ $\leq r_1+\ldots+r_k + s_1+\ldots+s_k$.
\end{proposition}

\begin{proof}We prove the case $k=2$, leaving the straightforward generalization to $k>2$ to the reader. Weight is additive over tensor products, so that $F \otimes G$ has weight $\leq s_1+s_2$ by part (i) of Proposition \ref{proptensorgeneweight}. By Lemma \ref{FIrepstab}, $F$ and $G$ are generated in degree $\leq r_1+s_1$ and $\leq r_2+s_2$ respectively, and since degree of generation is additive by part (ii) of Proposition \ref{proptensorgeneweight}, $F \otimes G$ is generated in degree $\leq r_1+r_2+s_1+s_2$.

This leaves stability degree. Let the surjectivity degree be the range where $T: \Phi_k(F)_n \to \Phi_k(F)_{n+1}$ becomes a surjection. Then it suffices to prove that the surjectivity degree of $F \otimes G$ is $\leq s_1+s_2$ and that $\dim \Phi_a(F \otimes G)$ is constant for $n \geq \max(r_1+s_1,r_2+s_2)$. We first prove the statement about the stabilization of the dimensions of the $\Phi_a$. Let $c_{n,\lambda}(F)$ denote the multiplicity of $V(\lambda)_n$ in $F_n$. We note that \begin{align*}\dim \Phi_0(F \otimes G)_n &= \dim(F_n \otimes G_n)_{\fS_n} = \dim \mr{Hom}_{\fS_n}(F_n \otimes G_n,\bQ) \\
&= \dim \mr{Hom}_{\fS_n}(F_n,\mr{Hom}_\bQ(G_n,\bQ)) = \sum_{\lambda} c_{n,\lambda}(F)c_{n,\lambda}(G)\end{align*} 
where we used in order: (i) the definition of $\Phi_0$, (ii) that a $\fS_n$-invariant function $F_n \otimes G_n \to \bQ$ is uniquely determined by its value on coinvariants, (iii) the tensor-hom adjunction, (iv) Schur's lemma and the fact that $\mr{Hom}_\bQ(G_n,\bQ)$ is non-canonically isomorphic to $G_n$  since the irreducible representations of symmetric groups are self-dual. We conclude that $\dim \Phi_0(F \otimes G)$ stabilizes when the multiplicities of $F_n$ and $G_n$ have both stabilized. By Lemma  \ref{FIrepstab} we have that $F$ and $G$ are uniformly representation stable with ranges $\geq r_1+s_1$ and $\geq r_2+s_2$ respectively. This implies that their multiplicities have stabilized by then, so $\dim \Phi_0(F \otimes G)_n$ stabilize for $n \geq \max(r_1+s_1,r_2+s_2)$.

Next we consider $\Phi_a(F \otimes G)_n = (\mr{Res}_{\fS_n}^{\fS_{n+a}}F_{n+a} \otimes \mr{Res}_{\fS_n}^{\fS_{n+a}}G_{n+a})_{\fS_n}$. By the argument given before, it suffices to find out when the multiplicities of $\mr{Res}_{\fS_n}^{\fS_{n+a}}F_{n+a}$ stabilize. Our notation implies that $F_{n+a} = \bigoplus_{\lambda} c_{n+a,\lambda}(F)V(\lambda)_{n+a}$. We have that $V(\lambda)_{n+a}$ for $\lambda = (\lambda_1,\ldots,\lambda_m)$ is given by $V_{\tilde{\lambda}}$ for $\tilde{\lambda} = (n+a-|\lambda|,\lambda_1,\ldots,\lambda_m)$ whenever that makes sense. The branching rule for restriction of $\fS_n$-representations  -- also known as Pieri’s rule, e.g. see A.7 of \cite{FH} -- says, in terms of Young tableaux, that $\mr{Res}_{\fS_n}^{\fS_{n+a}} V(\lambda)_{n+a}$ is given by all ways of removing $a$ boxes from the tableau corresponding to $\tilde{\lambda}$, such that each step we have a tableau.

This means that the branching of $V(\lambda)_{n+a}$ is independent of $n$ as long as $n+a-|\lambda| \geq \lambda_1+a$, i.e. $n \geq |\lambda| + \lambda_1$. Stability degree $\leq r$ implies that $\lambda_1 \leq r$ by Proposition 3.2.8 of \cite{CEF} and by definition weight $\leq s$ means that $|\lambda| \leq s$. This implies that the branching stabilizes when $n \geq r_1+s_1$ and $n \geq r_2+s_2$ respectively. As the multiplicities stabilize for $n \geq r_1+s_1$ and $n \geq r_2+s_2$ respectively, we conclude that $\dim \Phi_a(F \otimes G)$ stabilizes when $n \geq \max(r_1+s_1,r_2+s_2)$.

Next we prove the statement about the surjectivity degree. We start by noting that $\Phi_a$ is exact, because taking coinvariants by finite groups is exact in characteristic zero. Thus it suffices to prove the result after replacing $F$ and $G$ by $\FI$-$\bQ$-modules surjecting onto them. There is a natural surjection $M(H_0(F)) \to F$, by the rational version of Proposition 2.3.5 of \cite{CEF}, and so we can replace $F$ and $G$ by the free $\FI$-$\bQ$-modules $M(H_0(F))$ and $M(H_0(G))$ respectively. These are $\FI\#$-$\bQ$-modules and since $F$ and $G$ have weight $\leq s_1,s_2$ respectively, so will have $M(H_0(F))$ and $M(H_0(G))$ by Proposition 3.2.4 of \cite{CEF}. Then $M(H_0(F)) \otimes M(H_0(G))$ will be an $\FI\#$-$\bQ$-module of weight $\leq s_1+s_2$ by Proposition \ref{proptensorgeneweight}. By Proposition \ref{propdegweight} it will thus have stability degree $\leq s_1+s_2$ and as a consequence surjectivity degree $\leq s_1+s_2$.\end{proof}

\section{Diagrams of $1$-connected spaces} In this section we discuss a context in which one can make sense of Whitehead and Postnikov towers of diagrams of spaces. These towers are our main tool for moving between cohomology and homotopy groups, but we need to be careful, because our examples are not co-$\FI$-based spaces, though they are co-$\FI$-spaces. In this section, let $\I$ denote a small category. In the applications we have in mind, $\I$ is either $\FI$, $\FI\#$ or $\FI(G)$, a category whose definition we recall in Section \ref{secG}.

\label{secbasepoints}

\subsection{Variations on co-$\I$-spaces} We will now define three different notions of co-$\I$-spaces. Let $\cat{Top}$ denote the category of topological spaces and continuous maps, and let $\cat{hTop}$ denote the homotopy category of $\cat{Top}$, that is, the objects are homotopy types of spaces and the morphisms are homotopy classes of maps.

\begin{definition}\begin{enumerate}[(i)]
\item A \emph{co-$\I$-space} $X$ is a functor $\I^\mr{op} \to \cat{Top}$.
\item A \emph{semistrict co-$\I$-space $X$} consists of a space $X_k$ for each object $k$ of $\I$ and for each morphism $\tau: k \to k'$ a map $\tau^*: X_{k'} \to X_k$. These must have the property that if $\sigma: k \to k'$ and $\tau: k' \to k''$ are morphisms, then $\sigma^* \circ \tau^*$ is homotopic to $(\tau \circ \sigma)^*$ (but these homotopies are not part of the data).
\item A \emph{homotopy co-$\I$-space} $X$ is a functor $\I^\mr{op} \to \cat{hTop}$.
\end{enumerate}
\end{definition}

This notation is slightly misleading as homotopy co-$\I$-spaces and semistrict co-$\I$-space are not actually co-$\I$-spaces. Instead, the relation between these definitions is as follows. A semistrict co-$\I$-space can be thought of as a truncation of a homotopy coherent co-$\I$-space. Any co-$\I$-space is a semistrict co-$\I$-space, since then $\sigma^* \tau^*$ is equal to $(\tau \circ \sigma)^*$. Similarly, any semistrict co-$\I$-space gives rise to a homotopy co-$\I$-space by passing to homotopy types of spaces and homotopy classes of maps. We thus have inclusions:
\[\text{co-$\I$-spaces} \subset \text{semistrict co-$\I$-spaces} \subset \text{homotopy co-$\I$-spaces}\]
All of these come with a corresponding notion of map.

\begin{definition}\begin{enumerate}[(i)]
\item A \emph{map of co-$\FI$-spaces} $f: X \to Y$ is a natural transformation of functors $\FI^\mr{op} \to \cat{Top}$.
\item A \emph{semistrict map} $f: X \to Y$ between semistrict co-$\FI$-spaces is a sequence of maps $f_k: X_k \to Y_k$ such that
\begin{enumerate}[(a)]
\item the equation $f_k \circ \tau^*_X = \tau^*_Y \circ f_{k'}$ holds, i.e. the following diagram commutes
\[\xymatrix{X_{k'} \ar[r]^{\tau^*_X} \ar[d]_{f_{k'}} & X_k \ar[d]^{f_{k}} \\
Y_{k'} \ar[r]_{\tau^*_Y} & Y_k}\]
\item for $\sigma: k \to k'$ and $\tau: k' \to k''$ one can choose the homotopies $\sigma^* \tau^* \sim (\tau \circ \sigma)^*$ for $X$ and $Y$ to be compatible with $f$. More precisely, this condition means that there exist maps $H^X_{\sigma,\tau}: [0,1] \times X_{k''} \to X_k$ and $H^Y_{\sigma,\tau}: [0,1] \times Y_{k''} \to Y_k$ such that the restrictions to $0$ and $1$ give $\sigma^* \tau^*$ and $(\tau \circ \sigma)^*$ respectively, and we have $f_k \circ H^X_{\sigma,\tau} = H^Y_{\sigma, \tau} \circ (\mr{id}_{[0,1]} \times f_{k''})$; that is, the following diagram commutes
\[\xymatrix{[0,1] \times X_{k''} \ar[r]^(.65){H^X_{\sigma,\tau}} \ar[d]_{\mr{id}_{[0,1]} \times f_{k''}} & X_k \ar[d]^{f_k} \\
[0,1] \times Y_{k''} \ar[r]_(.65){H^Y_{\sigma,\tau}} & Y_k}\]
\end{enumerate}

\item A \emph{map of homotopy co-$\FI$-spaces} $f: X \to Y$ is a natural transformation of functors $\FI^\mr{op} \to \cat{hTop}$.
\end{enumerate}\end{definition}

%Because it is conceptually easier to take homotopy groups and cohomology groups of actual spaces instead of homotopy types, we will only consider the homotopy groups and cohomology groups of semistrict co-$\FI$-spaces.

We will now discuss homotopy groups and cohomology groups of co-$\I$-spaces, though the former is only relevant in the  $1$-connected setting.

\begin{definition}A semistrict co-$\I$-space $X$ is said to be \emph{$1$-connected} if each $X_k$ is.\end{definition}

\begin{definition}For a ring $R$, an \emph{$\I$-$R$-module} is a functor from $\I$ to the category of $R$-modules. A \emph{co-$\I$-$R$-module} is a functor from $\I^\mr{op}$ to the category of $R$-modules.\end{definition}

\begin{proposition}Let $X$ be a semistrict co-$\I$-space. 
\begin{enumerate}[(i)]
\item The cohomology groups of $H^i(X;R)$ are $\I$-$R$-modules. 
\item Pick arbitrary base points $*_k \in X_k$. If $X$ is $1$-connected, then the homotopy groups $\pi_i(X_k,*_k)$ can be made into an co-$\I$-$\Z$-module without making any additional choices.\end{enumerate}\end{proposition}

\begin{proof}Each morphism $\tau: k \to k'$ induces a map $H^i(\tau^*): H^i(X_{k};R) \to H^i(X_{k'};R)$ and since taking cohomology is homotopy-invariant, these maps satisfy the desired equation $H^i(\tau^*) \circ H^i(\sigma^*) = H^i((\tau \circ \sigma)^*)$.

For homotopy groups we need to be careful about the base points. If $X_k$ is $1$-connected there is a canonical isomorphism $\phi_{x_k}: \pi_i(X_k,x_k) \to \pi_i(X_k,*_k)$ for any $x_k \in X_k$. Thus each $\tau: k \to k'$ induces a map 
\[\pi_i(\tau^*) := \phi_{\tau^*(*_{k'})} \circ (\tau^*)_*: \pi_i(X_{k'},*_{k'}) \to \pi_i(X_k,\tau^*(*_{k'})) \to \pi_i(X_k,*_k)\]
These maps satisfy $\pi_i((\tau \circ \sigma)^*) = \pi_i(\sigma^*) \circ \pi_i(\tau^*)$.
\end{proof}

%In fact, the proof of part (ii) of the previous proposition shows we only need each fundamental group to act trivially on $\pi_i$.

\subsection{Fibers of semistrict co-$\I$-spaces} In this subsection we consider homotopy fibers of semistrict maps. This involves considering path-lifting functions for Hurewicz fibrations. Remember that a Hurewicz fibration $f: E \to B$ has dotted lifts in any diagram
\[\xymatrix{A \times \{0\} \ar[r] \ar[d]_{i_0} & E \ar[d]^f \\
A \times [0,1] \ar[r] \ar@{.>}[ru] & B}\]

The path space $B^I$ consists of continuous maps $[0,1] \to B$ in the compact-open topology. Let $E \times_f B^I$ be the space of pairs $(e,\gamma)$ such that $f(e) = \gamma(0)$. We set $A = E \times_f B^I$, the map $A \times \{0\} \to E$ the projection to $E$ and the map $A \times [0,1] \to B$ to be given by $(e,\gamma,t) \mapsto \gamma(t)$. A lift $A \times [0,1] \to E$ then corresponds to a map $\Phi: E \times_f B^I \to E^I$ called a path-lifting function. Path-lifting functions are unique up to homotopy.

Let $X$ be a semistrict co-$\I$-space. Then the path space $X^I$ can be made into a semistrict co-$\I$-space by setting $\tau^*_{X^I}(\gamma)(t) = \tau^*_X(\gamma(t))$. In fact, this construction gives a semistrict co-$\I$-space structure when mapping any space into a semistrict co-$\I$-space.

\begin{proposition}\label{propsemistrictfiber} Let $E \to B$ be a semistrict map of semistrict co-$\I$-spaces that is levelwise a Hurewicz fibration for which there exist path lifting functions $\Phi_k: E_k \times_f B_k^I \to E_k^I$ such that $\tau^*_{E^I} \Phi_k = \Phi_{k'} \circ (\tau^*_E \times \tau^*_{B^I})$ for each morphism $\tau: k \to k'$:
\[\xymatrix{E_{k'} \times_f B_{k'}^I \ar[r]^(.65){\Phi_{k'}} \ar[d]_{\tau^*_E \times \tau^*_{B^I}} & E_{k'}^I \ar[d]^{\tau^*_{E^I}}\\
E_k \times_f B_k^I \ar[r]_(.65){\Phi_k} & E_k^I}\]
Furthermore suppose that each $B_k$ is $1$-connected and pick an arbitrary base point $b_k \in B_k$. 

Then $F_k := f^{-1}(b_k)$ can be given the structure of a semistrict co-$\I$-space. This depends only on the choice of base points $b_k$, the lifting functions $\Phi_k$ and paths $\eta_\sigma$. Up to objectwise homotopy equivalence of co-$\I$-spaces, it is independent of the choice of base points and paths.\end{proposition}

\begin{proof}We have that $\tau^*_E$ maps $f^{-1}(b_{k'})$ to $f^{-1}(\tau^*_B(b_{k'}))$, while we would like a map $f^{-1}(b_{k'}) \to f^{-1}(b_k)$. To do this, we use the path-lifting function to lift along a path from one fiber to another. More precisely, we begin by picking for each morphism $\tau: k \to k'$ a path $\eta_\tau$ in $B_k$ from $\tau^*_B(b_{k'})$ to $b_k$. By restricting a path lifting function $\Phi: E_k \times_f B_k^I \to E_k^I$ to $F_{\tau^*_B(b_{k'})} \times \{\eta_\tau\}$ and evaluating at $t=1$, which we denote by $\phi_{\eta_\tau}$ we can map $f^{-1}(\tau_B^*(b_{k'}))$ to $f^{-1}(b_k)$. Thus we define
\[\tau^*_F = \phi_{\eta_\tau} \circ \tau^*_E\]

Since we only picked base points, paths and lifting functions, this structure only depends on those choices. Since any two choices of base points or paths can be connected by paths or homotopies respectively, these two choices do not matter up to objectwise homotopy equivalence. However, even though path-lifting functions are unique up to homotopy, we do not claim that compatible path lifting functions are.

We need to give homotopies $H^F_{\sigma,\tau}$ from $\sigma^*_F \tau^*_F$ to $(\tau \circ \sigma)^*_F$. These maps are explicitly given by
\[\sigma^*_F \tau^*_F = \phi_{\eta_\sigma} \circ  \sigma^*_E \circ \phi_{\eta_\tau} \circ \tau^*_E\qquad \text{ and } \qquad (\tau \circ \sigma)^*_F = \phi_{\eta_{\tau \circ \sigma}} \circ (\tau \circ \sigma)^*_E\]

We first claim that by compatibility of the path lifting functions, in the first line we have 
\[\phi_{\eta_\sigma} \circ  \sigma^*_E \circ \phi_{\eta_\tau} \circ \tau^*_E = \phi_{\eta_\sigma} \circ \phi_{\sigma^*_B(\eta_\tau)} \circ \sigma^*_E \circ \tau^*_E\]

To see that this equation holds, we only need to show that $\sigma^*_E \circ \phi_{\eta_\tau} = \phi_{\sigma^*_B(\eta_\tau)} \circ \sigma^*_E $. It suffices to show that $\sigma^*_{E^I} \circ \Phi|_{F_{\tau^*_B(b_{k'})} \times \{\eta_\tau\}} = \Phi|_{F_{\sigma^*_B \tau^*_B(b_{k'})} \times \{\sigma^*(\eta_\tau)\}} \circ \sigma^*_E$, which is just a restriction of the equation $\sigma^*_{E^I} \circ \Phi = \Phi \circ ( \sigma^*_E \times \sigma^*_{B^I})$. It thus suffices to show that 
\[\phi_{\eta_\sigma} \circ \phi_{\sigma^*_B(\eta_\tau)} \circ \sigma^*_E \circ \tau^*_E \sim \phi_{\eta_{\tau \circ \sigma}} \circ (\tau \circ \sigma)^*_E\]

\begin{figure}[t]
\includegraphics[width=15cm]{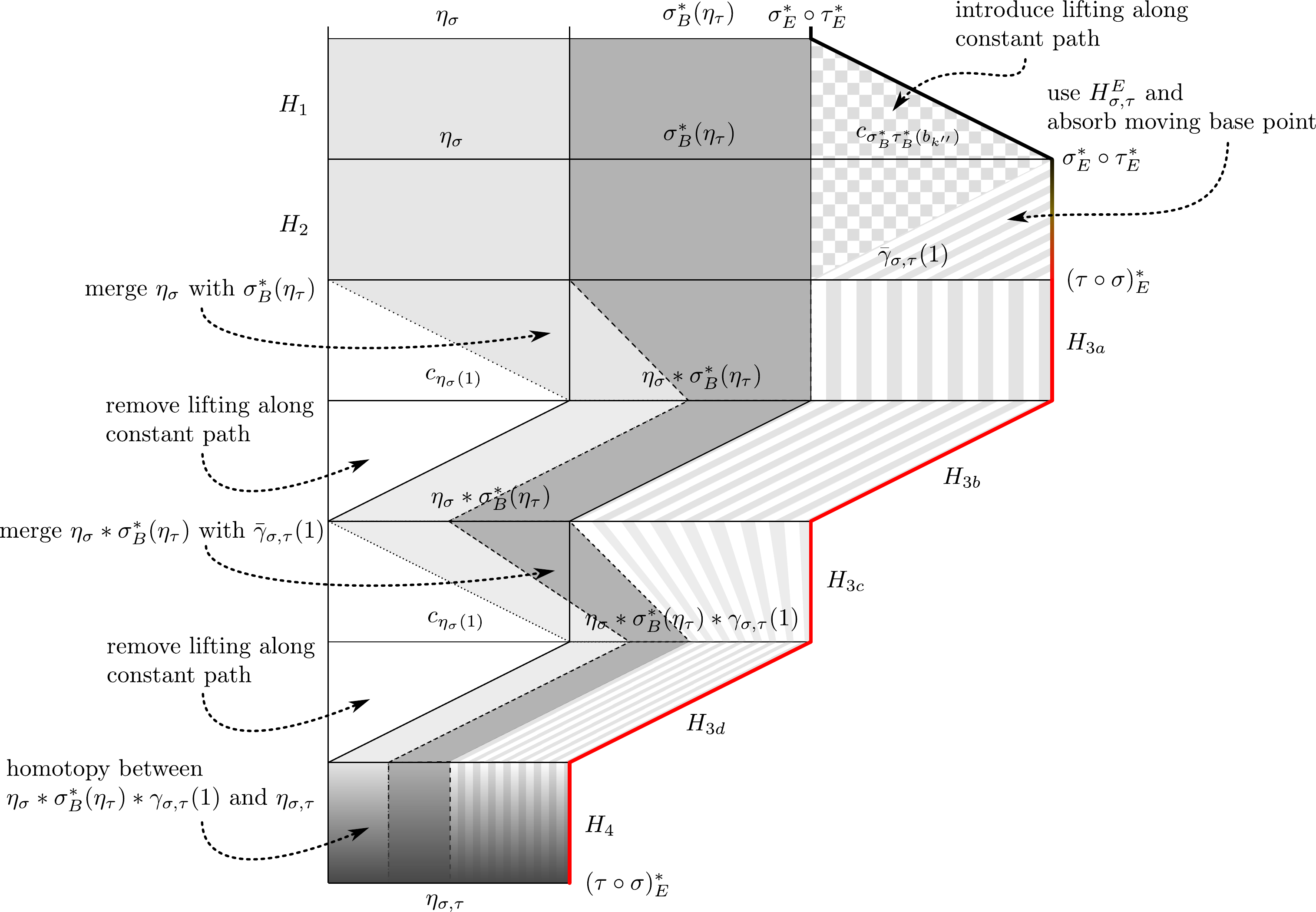}
\caption{A picture of the homotopy from $\phi_{\sigma^*(\eta_\tau)} \circ \sigma^*_E \circ \tau^*_E$ to $\phi_{\eta_{\tau \circ \sigma}} \circ (\tau \circ \sigma)^*_E$.}
\label{fighomotopy}
\end{figure}

Now the intuition is as follows: remark that $\sigma_E^* \circ \tau^*_E$ and $(\tau \circ \sigma)_E^*$ are homotopic, so we can homotope them to be equal. This comes at the expense of introducing another path-lifting function, since this homotopy might not preserve the base point. The resulting maps differ by path-lifting along arcs with the same start- and end-points, so by $1$-connectedness we can homotope one to the other. The remainder of this proof is concerned with given the details of this argument.

To define the maps $H^F_{\sigma,\tau}$, we start by picking compatible homotopies $H^E_{\sigma,\tau}$ and $H^B_{\sigma,\tau}$. We would then like to apply $H^E_{\sigma,\tau}$ to $\sigma^*_E \circ \tau^*_E$, but this might move in the $B$-direction, we need to insert an additional path lifting. Let $c_b$ denote the constant path at $b$. Our first homotopy $H_1$ is given by
\[H_1(t)(f) = \phi_{\eta_\sigma} \circ \phi_{\sigma^*_B(\eta_\tau)} \circ \mr{ev}_t \Phi(c_{\sigma^*_B \tau^*_B(b_{k''})}) \circ \sigma^*_E \circ \tau^*_E(f)\]
starting at $\phi_{\eta_\sigma} \circ \phi_{\sigma^*_B(\eta_\tau)} \circ \sigma^*_E \circ \tau^*_E$ and ending at $\phi_{\eta_\sigma} \circ \phi_{\sigma^*_B(\eta_\tau)} \circ \phi_{c_{\sigma^*_B \tau^*_B(b_{k''})}} \circ \sigma^*_E \circ \tau^*_E$.

Let $\gamma_{\sigma,\tau}$ denote the path in $B_k$ given by $s \mapsto H^B_{\sigma,\tau}(s)(\sigma^*_B \tau^*_B(b_{k''}))$, which starts at $\sigma^*_B \tau^*_B(b_{k''})$ and ends at $(\tau \circ \sigma)^*_B(b_{k''})$. Let $\bar{\gamma}_{\sigma,\tau}(t)$ be the path obtained by running $\gamma_{\sigma,\tau}(t)$ backwards from time $t$ to $0$. Note that $\bar{\gamma}_{\sigma,\tau}(0) = c_{\sigma^*_B \tau^*_B(b_{k''})}$.

We then define a second homotopy
\[H_2(t)(f) = \phi_{\eta_\sigma} \circ \phi_{\sigma_B^*(\eta_\tau)} \circ \phi_{\bar{\gamma}_{\sigma,\tau}(t)} \circ H^E_{\sigma,\tau}(t)(f)\]
which starts at $H_1(1)(f)$ and ends at $\phi_{\eta_\sigma} \circ \phi_{\sigma_B^*(\eta_\tau)} \circ \phi_{\bar{\gamma}_{\sigma,\tau}(1)} \circ (\tau \circ \sigma)^*_E(f)$. Our next goal is to combine the path lifting functions into a single one. 

This can be done by the concatenation of the four following homotopies. Let $\eta_\sigma(s)$ be the path obtained by running $\eta_\sigma$ from time $s$ to $1$ and $\bar{\eta}_\sigma(s)*\sigma^*_B(\eta_\tau)$ be the path obtained by running $\sigma^*_B(\eta_\tau)$ and then $\eta_\sigma$ from time $0$ to $s$.
\[H_{3a}(t)(f) = \phi_{\eta_\sigma(t)} \circ \phi_{\bar{\eta}_\sigma(t)*\sigma^*_B(\eta_\tau)} \circ \phi_{\bar{\gamma}_{\sigma,\tau}(1)} \circ (\tau \circ \sigma)^*_E(f)\] 
At the end of $H_{3a}$ we have $\phi_{c_{\eta_\sigma(1)}} \circ \phi_{\eta_\sigma(1)*\sigma^*_B(\eta_\tau)} \circ \phi_{\bar{\gamma}_{\sigma,\tau}(1)} \circ (\tau \circ \sigma)^*_E(f)$ and we can get rid of the path-lifting along a constant path via homotopy $H_{3b}$ like $H_1$. Then $H_{3c}$ combines the two remaining path-liftings into a single one like $H_{3a}$ and $H_{3d}$ gets rid of the path-lifting along a constant path. The result, after these four homotopies, is $ \phi_{\eta_\sigma * \sigma^*_B(\eta_\tau) * \gamma_{\sigma,\tau}(1)} \circ (\tau \circ \sigma)^*_E(f)$.

Next a homotopy $h_{\sigma,\tau}$ from $\eta_\sigma * \sigma^*_B(\eta_\tau) * \gamma_{\sigma,\tau}(1)$ to $\eta_{\sigma,\tau}$ can be implemented, which exists by 1-connectivity of $B_k$.
\[H_4(t)(f) = \phi_{h_{\sigma,\tau}(t)} \circ (\tau \circ \sigma)^*_E(f)\]
The end result is $H_4(1)(f) = \phi_{\eta_{\sigma,\tau}} \circ (\tau \circ \sigma)^*_E(f)$.
\end{proof}

\begin{remark}\label{remfiberinclusion}Note that the inclusion of the fibers into the total spaces is not a semistrict map. It is however a map of homotopy co-$\I$-spaces.\end{remark}

%This construction is natural in the input.

\begin{proposition}\label{propsemistrictfibernat} Let $f: E \to B$ and $f': E' \to B'$ be two semistrict map of semistrict co-$\I$-spaces satisfying the condition of Proposition \ref{propsemistrictfiber}. Suppose we are given semistrict maps $\tilde{g}: E \to E'$ and $g: B \to B'$ with the following properties:
\begin{enumerate}[(i)]
\item The homotopies $H^E$, $H^{E'}$, $H^B$ and $H^{B'}$ can be picked compatibly, i.e. $\tilde{g} \circ H^E = H^{E'} \circ \tilde{g}$, $g \circ H^B = H^{B'} \circ g$, $f \circ H^E=H^B$ and $f' \circ H^{E'}=H^{B'}$.
\item The lifting functions $\Phi_k$ and $\Phi_k'$ can be picked compatibly, i.e. $\tau^*_{E^I} \Phi_k = \Phi_{k'} \circ (\tau^*_E \times \tau^*_{B^I})$, $\tau^*_{(E')^I} \Phi_k' = \Phi'_{k'} \circ (\tau^*_{E'} \times \tau^*_{(B')^I})$, $\tilde{g}^I \circ \Phi_k = \Phi_k' \circ (\tilde{g} \times g^I)$.
\end{enumerate}

Pick base points of $B$ and $B'$ such that $g$ is pointed. Then there is a semistrict map $F \to F'$. This map depends only on the choice of base points $b_k$, paths $\eta_\sigma$ and the compatible lifting functions $\Phi_k$ and $\Phi_{k'}$. Up to objectwise homotopy equivalence it is independent of the choices of base points and paths.
\end{proposition}

\begin{proof}From the conditions in the proposition, it directly follows that following the construction given in Proposition \ref{propsemistrictfiber} we have that $\tilde{g}_k$ to $F_k$ has image in $F'_k$ and satisfies $f_k \circ \tau_F^* = \tau_{F'}^* f_{k'}$. If one additionally uses $g$ to map the paths $\eta_\tau$ from $B_k$ to $B'_k$, and similarly for the homotopies $h_{\sigma,\tau}$, then the $H^F_{\sigma,\tau}$ will be compatible with the $H^{F'}_{\sigma,\tau}$.
\end{proof}

We want to apply the Serre spectral sequence to the situation in Proposition \ref{propsemistrictfiber}. This requires some discussion of the naturality of Serre spectral sequences.

\begin{digression}[Naturality of Serre spectral sequences] Let $\pi: E \to B$ be a Serre fibration with $B$ path-connected. Let $b \in B$ be a base point and $F = \pi^{-1}(b)$ the fiber above $b$. Then the cohomological Serre spectral sequence is a given by
\[E^{p,q}_2 = H^p(B,\mathcal H^q(F)) \Rightarrow H^{p+q}(E)\]
where $\mathcal H^q(F)$ is a local coefficient system with fiber $H^q(F)$ (since all our spaces are $1$-connected, this will always be trivial in our setting).

This is natural in the following sense: let $\pi: E \to B$ and $\pi': E' \to B'$ be Serre fibrations and suppose that $B$ and $B'$ are path-connected and endowed with base points $b$ and $b'$. Let $F$ and $F'$ denote the fibers over the basepoints of $B$ and $B'$ respectively. Suppose we are given a commutative diagram
\[\xymatrix{E \ar[r]^g \ar[d]_\pi & E' \ar[d]^{\pi'} \\
B \ar[r]_f & B}\]
such that $f(b) = b'$, then there is a map of Serre spectral sequence, induced on $E^2$ and $E^\infty$ by the appropriate maps on cohomology. What happens if $f$ does not preserve the base point? There always exist a map from the Serre spectral sequence with base point at $f(b)$ to the one with base point at $b'$, which induces an isomorphism on the $E^2$-page and proceeding pages. The construction is as follows. There exists a path $\beta$ connecting $f(b)$ to $b'$ then one can consider the zigzag
\[\xymatrix{E' \ar[d] &  \ar[l]_{\mr{ev}_0} E'^I \ar[d] \ar[r]^{\mr{ev}_1} & E' \ar[d] \\
B' & \ar[l]^{\mr{ev}_0}  B'^I \ar[r]_{\mr{ev}_1} & B'}\]
where the path space $B'^I$ is pointed by $\beta$. Now all the maps on bases preserve the base point. The evaluations induce isomorphisms on the $E^2$-page and hence on all proceeding pages. Note that this a priori depends on the choice of path, but if both $B$ and $B'$ are $1$-connected, it is independent of that choice. To see this, one can use a Serre spectral for $E'^{I^2} \to B'^{I^2}$ based at a homotopy between $\beta$ and a second choice $\beta'$. 

We next identify the resulting maps on the $E^2$- and $E^\infty$-pages. For the base and total space they are given by the zigzag of isomorphisms
\[\xymatrix{H^*(B') \ar[r]^(.45){\mr{ev}_0^*} & H^*(B'^I) & \ar[l]_(.45){\mr{ev}_1^*} H^*(B')} \qquad
\xymatrix{H^*(E') \ar[r]^(.45){\mr{ev}_0^*} & H^*(E'^I) & \ar[l]_(.45){\mr{ev}_1^*} H^*(E')}\]
However, both $B'^I$ and $E'^I$ retract onto the constant loops and then both evaluation maps are the identity, so that the composition of both isomorphisms is the identity. What about the map on the cohomology of fibers, i.e. the isomorphism $H^*(F'_{f(b)}) \cong H^*(F'_{b'})$? The zigzag tells us its given by the zigzag of isomorphisms
\[\xymatrix{H^*(F'_{f(b)}) \ar[r]^{\mr{ev}_0^*} & H^*(F'^I_\beta) & \ar[l]_{\mr{ev}_1^*} H^*(F'_{b'})}\]

We would get an explicit isomorphism $H^*(F'_{b'}) \to H^*(F'_{f(b)})$ if we had a homotopy inverse to $\mr{ev}_0: F'^I_\beta \to F'_{f(b)}$. Such an inverse is provided by restricting a path lifting function $\Phi$ to a map $\Phi_{\bar{\beta}}$ from $F'_{f(b)}$ to lifts of $\bar{\beta}$, where $\bar{\beta}$ is the reversed path. Then the explicit isomorphism given by $\mr{ev}_1^* \circ \Phi_{\bar{\beta}}^*$. This is independent of the choice of path lifting function, since they are unique up to homotopy.\end{digression}

\begin{proposition}Suppose we are in the setting of Proposition \ref{propsemistrictfiber}. Then the cohomological Serre spectral sequences $H^p(B_k,\mathcal H^q(F_{k})) \Rightarrow H^{p+q}(E_k)$ assemble into a spectral sequence of $\I$-$\bZ$-modules. Here the $\I$-$\bZ$-module structure on the $E^2$-page and $E^\infty$-page is induced by the semistrict co-$\I$-space structures of $E_k$, $B_k$ and $F_k$ respectively.
\end{proposition}

\begin{proof}The Serre spectral sequence comes from the fibration $E_k \to B_k$ of semistrict co-$\I$-spaces, based at $*_k$. Each morphism $\tau: k \to k'$ of $\I$ induces a map $\tau^*$ of fibrations, which might not preserve this base point. However, by the previous discussion it does induce a unique map of spectral sequences from $H^p(B_{k'},\mathcal H^q(F_{k'})) \Rightarrow H^{p+q}(E_{k'})$ to $H^p(B_{k},\mathcal H^q(F_{k})) \Rightarrow H^{p+q}(E_{k})$, where the map on the homology of the base and total space is induced by $\tau^*_E$ and $\tau^*_B$. On the fiber the map is given by restriction of $\tau^*_E$ followed by a path lifting function, which is exactly how we constructed the semistrict co-$\I$-space structure on $F_k$. By the properties of semistrict co-$\I$-spaces, this tells us that $\sigma^* \circ \tau^*$ and $(\tau \circ \sigma)^*$ are equal on $E^2$ and hence on all subsequent pages.\end{proof}

\subsection{Homotopy fibers} We will now discuss the most important and in a sense universal way in which the situation of Proposition \ref{propsemistrictfiber} arises.

Let $X \to Y$ be a semistrict map between semistrict co-$\I$-spaces. Pick a base point $*_k \in Y_k$. The standard path fibration construction gives a semistrict map of semistrict co-$\I$-spaces with properties as in the previous proposition. Here our homotopy fibers are given by standard construction
\[\mr{hofib}(f_k,*_k) = \{(x,\gamma) \in X_k \times \mr{Map}([0,1],Y_k)|\gamma(0) = f_k(x) \text{ and } \gamma(1) = *_k\}\]

Note that projection to the first component gives a map $p_1: \mr{hofib}(f_k,*_k) \to X_k$.

\begin{proposition}\label{prophofibstruct} If each $Y_k$ is $1$-connected, the sequence $\{\mr{hofib}(f_k,*_k)\}_{k \geq 0}$ can be given the structure of a semistrict co-$\I$-space, which we denote by $\mr{hofib}(f)$, such that the map $p_1: \mr{hofib}(f) \to X$ is a semistrict homotopy equivalence. This structure is independent of any choice up to objectwise homotopy equivalence.\end{proposition}

\begin{proof}We replace $X_k \to Y_k$ with the Hurewicz fibration $\tilde{X}_k = X_k \times_f Y_k^I \to Y_k$. The $\tilde{X}_k$ assemble into a semistrict co-$\I$ space by using the relevant maps pointwise. That is, we set
\[\tau^*_{\tilde{X}}(x,s \mapsto \gamma(s)) = (\tau^*_X(x),s \mapsto \tau^*_{Y}(\gamma)(s))\]
and setting $H^{\tilde{X}}_{\sigma,\tau}(x,s \mapsto \gamma(s)) = (H^X_{\sigma,\tau}(x),s \mapsto H^Y_{\sigma,\tau}(\gamma)(s))$. It is clear that with these definitions that $\pi_1: \tilde{X} \to X$ and $\tilde{f}: = \mr{ev}_1: \tilde{X} \to Y$ are a semistrict maps.

We can apply Proposition \ref{propsemistrictfibernat} if we can find compatible lifting functions $\Phi_k: \tilde{X}_k \times_{\tilde{f}_k} Y_k^I \to \tilde{X}_k^I$. But the standard lifting functions satisfy the condition for compatibility, as they are given by
\[\Phi_k(x,\gamma,\gamma') = (x,\gamma * \gamma')\]
where $*$ denotes concatenation of loops. These lifting functions are furthermore canonically defined for any $f: X \to Y$, so this co-$\I$-space structure is independent of choices up to a zigzag of maps of co-$\I$-spaces that are objectwise homotopy equivalences.
\end{proof}

\subsection{Postnikov and Whitehead towers} We next develop the theory of Postnikov and Whitehead towers of semistrict co-$\I$-spaces. For this we recall the definitions for individual spaces, starting with the Postnikov tower. 

The following definition is given in \cite[Chapter VI, Section 2]{GoJ}. It uses the \emph{coskeleton} of a simplicial set, which is the composite $R \circ L$, where $R: \mr{Fun}(\Delta^\mr{op},\cat{Set}) \to \mr{Fun}(\Delta^\mr{op}_{\leq n},\cat{Set})$ is the restriction and $L$ is its left adjoint. For a simplicial set $S$, we denote the $n$th coskeleton by $\mr{cosk}_n(S)$. For a space $X$, let $\mr{Sing}(X)$ denote the simplicial set of singular simplicies of $X$. 

\begin{definition}The \emph{$n$th stage of the Postnikov tower $P_n(X)$} is defined as
\[P_n(X) = |\mr{cosk}_n \mr{Sing}(X)|\]
\end{definition}

There is a functorial weak equivalence $|\mr{Sing}(X)| \to X$, so without loss of generality we can replace $X$ with $|\mr{Sing}(X)|$. Then there are canonical maps $p_n: X \to P_n(X)$ and $p^n_{n-1}: P_n(X) \to P_{n-1}(X)$. The result is a tower under $X$, called the \emph{Postnikov tower} and this construction is functorial in $X$. If $X$ has a base point $*$, then this induces a base point $*$ in $P_n(X)$. We have that each $P_n(X)$ is $n$-coconnected with homotopy groups given by
\[\pi_i(P_n(X),*) = \begin{cases}\pi_i(X,*) & \text{if $i \leq n$} \\
0 & \text{if $i > n$}\end{cases}\]
and the maps $p_n$ and $p^n_{n-1}$ behave as expected on homotopy groups. The Whitehead tower is then defined as the homotopy fiber of the Postnikov tower.

\begin{definition}Suppose that $X$ has a base point $*$, which gives a base point $* \in P_n(X)$. The \emph{$n$th stage of the Whitehead tower $W_n(X)$} is defined by
\[W_n(X) = \mr{hofib}(p_n,*)\]
\end{definition}

There is a canonical inclusion $W_n(X) \to X \times_{p_n} P_n(X)$ and the composition with the projection $p_1: X \times_{p_n} P_n(X) \to X$ gives a canonical map $w_n: W_n(X) \to X$. There is also a canonical map
\[w^n_{n-1}: W_n(X) \to W_{n-1}(X)\]

This gives a functorial tower over $X$, as soon as one has picked a base point in $X$, which we call the \emph{Whitehead tower}. We have that each $P_n(X)$ is $n$-connected and its homotopy groups are given by
\[\pi_i(W_n(X),*) = \begin{cases}0 & \text{if $i \leq n$} \\
\pi_i(X,*) & \text{if $i > n$}\end{cases}\]
The maps $w_n$ and $w^n_{n-1}$ behave as expected on homotopy groups.

\begin{proposition}
\label{PWisSemiStrict}
\begin{enumerate}[(i)]
\item The Postnikov tower $P_n(X)$ of semistrict co-$\I$-space $X$ can be made into a tower of semistrict co-$\I$-spaces under $X$ and semistrict maps between them.
\item Pick a base point $*_k$ in $X_k$ for all objects $k$ of $\I$. The Whitehead tower $W_n(X)$ of a $1$-connected semistrict co-$\I$-space $X$ can be made into a tower of semistrict co-$\I$-spaces over $X$ and semistrict maps between them.
\end{enumerate}\end{proposition}

\begin{proof}The first claim follows directly from functoriality of the Postnikov tower construction. For the second claim, we recall that the levels of the Whitehead tower are defined as homotopy fibers of the maps $X \to P_n(X)$. If we use the same base points $*_k \in X_k$, paths $\eta_\tau$ from $\tau^*(*_{k'})$ to $*_k$ in $X$ and the standard path lifting function, Proposition \ref{propsemistrictfiber} applies to give the $W_n(X_k)$ the structure of a semistrict co-$\I$-spaces. The map $W_n(X) \to W_{n-1}(X)$ is semistrict by Proposition \ref{propsemistrictfibernat}.
\end{proof}

A base point in $X$ endows $P_n(X)$ and $W_n(X)$ with canonical base points, so that the maps $p_{n}^{n-1}$ and $w_n^{n-1}$ become based, as do the maps to and from $X$. The next Proposition uses a functorial $n$-fold classifying space functor $B^n$.

\begin{proposition}
Let $X$ be a $1$-connected semistrict co-$\I$-space and pick a base point $*_k$ in each $X_k$. Then for all $i \geq 0$ and $n \geq 1$, there are natural isomorphisms of $\I$-$\Z$-modules $H^i(B^n \pi_n(X,*_k)) \cong H^i(\mr{hofib}(p_n^{n-1}))$ and $H^i(B^{n-1} \pi_n(X,*_k)) \cong H^i(\mr{hofib}(w_n^{n-1}))$.
\label{fiberisB}
\end{proposition}

\begin{proof}By Proposition \ref{propsemistrictfiber} both $\mr{hofib}(p_n^{n-1})$ and $\mr{hofib}(w_n^{n-1})$ are semistrict co-$\I$-spaces. Any functorial classifying space construction makes $B^n \pi_n(X_k,*_k)$ and $B^{n-1} \pi_n(X_k,*_k)$ into co-$\I$-spaces. It is clear that all four of these semistrict co-$\I$-spaces consists of Eilenberg-Mac Lane spaces on groups abstractly isomorphic to $\pi_n(X_k,*_k)$, but we will pin down the identifications of the homotopy groups. 

In the Postnikov case this isomorphism is equal to the one induced on $\pi_n$ by the zigzag 
\[X \to P_n X \leftarrow \mr{hofib}(P_n(X) \to P_{n-1}(X))\] of pointed maps. Naturality of the Postnikov construction and Proposition \ref{prophofibstruct} imply that on $\pi_n$ we get an identification of co-$\I$-$\bZ$-modules.

In the Whitehead case, the isomorphism will equal the one induced on $\pi_{n-1}$ by the map induced by 
\[\Omega X \to \Omega W_{n-1}(X) \to \mr{hofib}(W_n(X) \to W_{n-1}(X))\]
Here we have to be more careful; the maps are only maps of homotopy co-$\I$-spaces. This is however enough to conclude that we get on $\pi_{n-1}$ an identification of co-$\I$-$\bZ$-modules.

These identification give unique homotopy classes of maps 
\[B^n \pi_n(X_k,*_k) \to \mr{hofib}(P_n(X) \to P_{n-1}(X))\] \[B^{n-1} \pi_n(X_k,*_k) \to \mr{hofib}(W_{n}(X) \to W_{n-1}(X))\]
since maps between Eilenberg-Mac Lane spaces are determined up to homotopy by their effect on $\pi_n$. For each morphism $\tau: k \to k'$ in $\I$, $\tau^*$ acts the same on the homotopy groups of each of these four co-$\I$-spaces, these are actually maps of homotopy co-$\I$-spaces. An equivalence of homotopy co-$\I$-spaces induces isomorphisms of $\I$-$\bZ$-modules on cohomology.
\end{proof}

\section{The proof of the main theorem}
\label{secmain}

In this section we prove Theorem \ref{thmequiv}. This states that a co-$\FI$-space has finitely generated cohomology if and only if it has finitely generated dual (in the derived sense) homotopy groups. In the first subsection we prove this result with integral coefficients without making quantitative statements. In the second subsection we obtain a quantitative result with rational coefficients. 

\begin{convention}In this section we will no longer refer to the base points of homotopy groups. All semistrict co-$\FI$- or co-$\FI\#$-spaces that appear are $1$-connected and in that case one can pick arbitrary base points and deal with these as in Section \ref{secbasepoints}.\end{convention}

\subsection{Integral portion of Theorem \ref{thmequiv}} In this section we prove the integral part of Theorem \ref{thmequiv}. 

\begin{remark}\label{remintegralreqs} Similar results can be obtained for other properties of $\I$-$\bZ$-modules than of finite generation of $\FI$-$\bZ$-modules or $\FI\#$-$\bZ$-modules, e.g. finite generation of $\cat{FI}(G)$-$\bZ$-modules for certain groups $G$. All one needs is that such a property $P$ preserved under taking finite extensions, quotients, submodules or tensor products, and that the relevant variation of Proposition \ref{propemfg} holds. In particular, our proof of that proposition goes through if every $\I$-$\bZ$-module with property $P$ admits a resolution by projective levelwise free $\I$-$\bZ$-modules with property $P$.\end{remark}

\begin{comment}We start with a discussion of a model for the classifying space of abelian group.

\begin{digression}[The simplicial and toroidal models of $BA$] \label{disc_b} 

It is naturally homeomorphic to a space $B_{[0,1]}(A)$ given by configurations of unordered points in $(0,1)$ labeled by $A$, topologized such that (i) the labels add if points collide, (ii) points labeled by the identity can disappear, and (iii) points disappear if they hit $0$ or $1$. This is a topological abelian group by superposition of configurations.

There is a map $B_s A \cong B_{[0,1]} A \to B_t A$ which is also a group homomorphism. This is given by sending a configuration $\{(a_i,t_i)|a \in A,\,t_i \in (0,1)\}$ to $\sum_i a_i \otimes t_i$, where the sum is taken in $A \otimes_\bZ \bR$. This is natural in $A$. It induces an isomorphism on $\pi_1$ and hence a homotopy equivalence. This discussion guarantees that $B_s$, $B_{[0,1]}$ and $B_t$ are functorially homotopy equivalent when they make sense.
\end{digression}\end{comment}

If $A$ is any topological abelian group, then one can define a simplicial bar construction $B^\mr{simp} A$. This is the geometric realization $|B_\bullet A|$ of the simplicial space $B_\bullet A$ given by $B_k A = A^{k}$. The face maps $d_i: B_k A \to B_{k-1} A$ for $0 \leq i \leq k$ are given by
\[d_i(a_1,\ldots,a_k) = \begin{cases} (a_2,\ldots,a_k) & \text{if $i=0$} \\
(a_1,\ldots,a_i+a_{i+1},\ldots,a_k) & \text{if $1 \leq i \leq k-1$} \\
(a_1,\ldots,a_{k-1}) & \text{if $i = k$}\end{cases}\]
and degeneracy maps by including identities. This is natural in continuous homomorphisms.  If $A$ was abelian, then $BA$ is again an abelian topological group and we can iterate the construction. If we use $B$ without decoration, it will mean $B^\mr{simp}$. Later, we will use other superscripts to denote other models of the bar construction. The following result is well-known, but we include a proof to stress its naturality.

\begin{lemma}\label{lemcohfree} Let $A$ be a finitely generated free abelian group. There is a natural isomorphism between $H^*(BA)$ and $\Lambda_\bZ \Homz{A}[1]$, the free graded-commutative algebra on $\Homz{A}$ in degree $1$.\end{lemma}

\begin{proof}If $A$ is a (discrete) finitely generated free abelian group, there is another model for the classifying space. This toroidal model $B^\mr{tor} A$ is the topological abelian group given by $(A \otimes_\bZ \bR)/A$. There is a natural continuous group homomorphism $B^\mr{simp} A \to B^\mr{tor} A$, given by sending $(\vec{t},a_1,\ldots,a_k) \in B^\mr{simp} A$ with $\vec{t} = (0\leq t_1 \leq \ldots \leq t_k \leq 1) \in \Delta^k$ to $\sum_i a_i \otimes t_i$, with sum taken in $A \otimes_\bZ \bR$. It induces an isomorphism on $\pi_1$ and hence is a homotopy equivalence. 

Note that $B^\mr{tor} A$ is a manifold, and we can identify $H^*(B^\mr{tor}A)$ with the subgroup of the de Rham cohomology $H^*_\mr{dR}(B^\mr{tor}A)$ having integer values on all integral homology classes. By averaging, $H^*_\mr{dR}(B^\mr{tor}A)$ can be identified with the exterior algebra on the cotangent space at $0$, placed in degree 1. This cotangent space can be identified with $\mr{Hom}_{\bZ}(A,\bR)$ and the integral-valued forms with $\Homz{A}$. All these identifications are natural.\end{proof}

\begin{proposition}\label{propemfg}
Let $F$ be a finitely generated $\FI$-$\Z$-module. For all $n \geq 1$ and all $i\geq 0$, the $\FI$-$\Z$-modules $H^i(B^n \Homz{F})$  and $H^i(B^n \Extz{F})$ are finitely generated. The same is true for $\FI\#$-$\Z$-modules.
\label{EM}
\end{proposition}

\begin{proof}We only prove the $\FI$-$\Z$-module case, as exactly the same proof can be used for $\FI\#$-$\Z$-modules by replacing each mention of $\FI$ with $\FI\#$. We start by separating $F$ into its free and torsion parts. There is a natural short exact sequence of $\FI$-$\bZ$-modules
\[0 \to F[\mr{tors}] \to F \to F/F[\mr{tors}] =: F[\mr{free}] \to 0\]

Because $F$ is finitely generated, both $F[\mr{tors}]$ and $F[\mr{free}]$ are finitely generated by Proposition \ref{propfgprops}. There are natural identifications 
\[\Homz{F} \cong \Homz{F[\mr{free}]} \quad \text{and} \quad \Extz{F} \cong \Extz{F[\mr{tors}]}\]
so it suffices to prove that $ H^i(B^n \Homz{F[\mr{free}]})$ and $H^i(B^n \Extz{F[\mr{tors}]})$ are finitely generated. 

We first prove that each cohomology group $H^i(B\Homz{F[\mr{free}]})$ is finitely generated. By Lemma \ref{lemcohfree} $H^*(B\Homz{F[\mr{free}]})$ is naturally isomorphic to $\Lambda^*_\bZ F[\mr{free}]$, the free exterior algebra. A free exterior algebra on a finitely generated $\FI$-$\bZ$-module in degree $1$ is finitely generated in each degree, because the exterior algebra is a quotient of the tensor algebra and tensor products preserve finite generation by Proposition \ref{propfgprops}.

For the cohomology of $B^n\Homz{F[\mr{free}]}$, we do an induction over $n$ using the natural geometric realization spectral sequence for cohomology of $B^n\Homz{F[\mr{free}]}$ and applying Corollary \ref{corssfg}. This spectral sequence is given by
\[H^q(B^{n-1} \Homz{F[\mr{free}]}^{\times p+1}) \Rightarrow H^{p+q}(B^n\Homz{F[\mr{free}]})\]
and we note that no $\mr{Tor}$-groups appear when applying the K\"unneth theorem, since all cohomology is free.

For $\mr{Ext}_\Z^1(F[\mr{tors}],\Z)$. we recall $F[\mr{tors}]$ is finitely generated, and thus by Proposition 2.3.5 of \cite{CEF} there is a surjection $M \to F[\mr{tors}]$ from a levelwise free finitely generated $\FI$-$\bZ$-module $M$. Let $N$ denote the kernel. This is levelwise free because submodules of free $\Z$ modules are free and is finitely generated by Proposition \ref{propfgprops}. This gives rise to a long exact sequence of co-$\FI$-$\bZ$-modules:
\begin{align*}0 &\to \Homz{F[\mr{tors}]} \to \Homz{M} \to \Homz{N} \to \\
&\to \Extz{F[\mr{tors}]} \to \Extz{M} \to \Extz{N} \to 0\end{align*}
and since $F[\mr{tors}]$ takes values in torsion abelian groups and $M$ and $N$ take values in free abelian groups, this reduces to a short exact sequence:
\[0 \to \Homz{M} \to \Homz{N} \to \Extz{F[\mr{tors}]} \to 0 \]
which in turn gives rise a fiber sequence of co-$\FI$-spaces upon taking classifying spaces:
\[B\Homz{M} \to B\Homz{N} \to B\Extz{F[\mr{tors}]}\]

The above sequence consists of short exact sequences of connected topological abelian groups. For an injective continuous homomorphism $H \hookrightarrow G$ of topological abelian groups. the cohomological Rothenberg-Steenrod spectral sequence \cite{RSt} is given by 
\[E_2^{p,q} = \mr{Tor}_{p,q}^{H^*(G)}(H^*(H),\bZ) \Rightarrow H^{p+q}(G/H)\]
and this is natural in commutative diagrams of continuous homomorphisms of topological groups. Thus we obtain a sequence a spectral sequence of $\FI$-$\bZ$-modules and by Corollary \ref{corssfg} it suffices to prove that $\mr{Tor}_{p,q}^{H^*(G)}(H^*(H),\bZ)$ is finitely generated for $G = B\Homz{M}$ and $H = B\Homz{N}$. The first part of this proof shows that $H^p(B\Homz{M})$ and $H^q(B\Homz{N})$ are finitely generated for all $p$ and $q$. Finite generation of $\mr{Tor}_{p,q}^{H^*(G)}(H^*(H),\bZ)$ now follows: using Proposition \ref{propfgprops} it suffices to remark that, in the bar resolution computing $\mr{Tor}$, all entries are finitely generated.

For the cohomology of $B^n\Extz{F[\mr{tors}]}$ we perform an induction over $n$ using Corollary \ref{corssfg} and the geometric realization spectral sequence for the cohomology of $B^n \Extz{F[\mr{tors}]}$. Here we need to use the full version of the K\"unneth theorem, which states that there is a natural short exact sequence
\[0 \to \bigoplus_{i+j = n} H^i(X) \otimes H^i(Y) \to H^n(X \times Y) \to \bigoplus_{i+j = n-1} \mr{Tor}_\bZ(H^i(X),H^j(Y)) \to 0\]

Thus $H^*(X \times Y)$ will consist of finitely generated $\FI$-$\bZ$-modules if $H^*(X)$ and $H^*(Y)$ do, as soon as we prove that $\mr{Tor}$ of two finitely generated $\FI$-$\bZ$-modules is finitely generated. But this follows because a finitely generated $\FI$-$\bZ$-module has a projective resolution by finitely generated $\FI$-$\bZ$-modules, which can be seen by combining Proposition 2.3.5 and Remark 2.2.A of \cite{CEF} with Proposition \ref{propfgprops}.\end{proof}

\begin{lemma}
Let $X$ be a $1$-connected semistrict co-$\FI$-space with $H^i(X)$ finitely generated for all $i$. Then for each $i \geq 0$, $n \geq 1$ and $k \geq  0$ the abelian group $H^i(W_n(X_k))$ is finitely generated. Furthermore, $H^i(W_n(X))$ is finitely generated as an $\FI$-$\bZ$-module for all $i$ and $n \geq 1$. The same is true for co-$\FI\#$-spaces.

\label{whitehead}
\end{lemma}

\begin{proof}We only give the proof of co-$\FI$-spaces and $\FI$-$\bZ$-modules. We remark that $w_1: W_1(X)\to X$ is a homotopy equivalence since $X$ is $1$-connected. This proves the result for $n = 1$. Suppose we have proven the result for $n-1$, then we will prove it for $n$. First of all, note that the Hurewicz map $\pi_n(X) \to H_n (W_{n-1}(X))$ is an isomorphism of co-$\FI$-$\bZ$-modules.

Proposition \ref{PWisSemiStrict} applies to the fibration 
\[\mr{hofib}(w^n_{n-1}) \to \tilde{W}_n(X) \to W_{n-1}(X)\]
where $\tilde{W}_n(X)$ denotes the result of replacing $W_n(X)$ by $W_n(X) \times_{w_n^{n-1}} W_{n-1}(X)^I$. Also recall there is a homotopy equivalence $\tilde{W}_n(X) \to W_n(X)$ of semistrict co-$\I$-spaces. We thus get a Serre spectral of $\FI$-$\bZ$-modules given by
\[E_2^{p,q} = H^p(W_{n-1} (X),H^q(\mr{hofib}(w^n_{n-1}))) \Rightarrow H^{p+q}(W_n(X))\]

In Proposition \ref{fiberisB} we proved that $H^q(\mr{hofib}(w^n_{n-1})) \cong H^q(B^{n-1}\pi_n(X))$ as an $\FI$-$\bZ$-module. So we would done by an application of Corollary \ref{corssfg} if we could prove that $H^q(B^{n-1}\pi_n(X))$ was finitely generated as an abelian group and an $\FI$-$\bZ$-module for all $q \geq 0$. Because the abelian groups are finitely generated, there is a natural short exact sequence of co-$\FI$-$\bZ$-modules:
\[0 \to \Extz{H^{n+1}(W_{n-1}(X))} \to H_n (W_{n-1}(X)) \cong \pi_n(X,*) \to \Homz{H^n (W_{n-1}(X))} \to 0\]
leading to a Serre spectral sequence of $\FI$-$\bZ$-modules
\[E_2^{p,q} = H^p(B^{n-1}(\Homz{H^n (W_{n-1}(X))}),H^q(B^{n-1} \Extz{H^{n+1}(W_{n-1}(X))}))\]
\[\Rightarrow H^{p+q}(B^{n-1}\pi_n(X))\]
By combining Corollary \ref{corssfg} and Proposition \ref{EM} we conclude that $H^q(B^{n-1}\pi_n(X))$ and $H^{p+q}(W_n(X))$ are both finitely generated as abelian groups and as $\FI$-$\bZ$-modules.\end{proof}

\begin{lemma}
\label{lempostnikov}
Let $X$ be a $1$-connected semistrict co-$\FI$-space with $\Homz{\pi_i(X)}$ and $\Extz{\pi_i(X)}$ finitely generated for all $i$. Then for each $i \geq 0$, $n \geq 1$ and $k \geq 0$ the abelian group $H^i(P_n(X_k))$ is finitely generated. Furthermore $H^i(P_n(X))$ is finitely generated as an $\FI$-$\Z$-module for all $i \geq 0$ and $n \geq 1$. The same is true for co-$\FI\#$-$\Z$-spaces.

\label{postnikov}
\end{lemma}

\begin{proof}We again only give the proof for co-$\FI$-spaces and $\FI$-$\bZ$-modules. We will proceed via an induction on $n$. Since $X$ is $1$-connected, $P_1(X)$ is contractible so the case $n=1$ is trivial. Suppose for the purposes of induction we have shown that $H^i(P_{n-1} (X))$ is finitely generated for all $i$. Using Proposition \ref{fiberisB}, we can identify the cohomology of the homotopy fiber of $p^{n-1}_n: P_n(X) \to P_{n-1}(X)$ with the cohomology of $B^n \pi_n(X)$. If we can show that $H^i(B^n \pi_n(X))$ is a finitely generated $\FI$-$\Z$-module for all $i$, the induction step will follow by applying Corollary \ref{corssfg} to the Serre spectral sequence.

Because we are dealing with finitely generated abelian groups, the short exact sequence \[0 \to \pi_n(X)[\mr{tors}] \to \pi_n(X) \to \pi_n(X)[\mr{free}] \to 0\] can naturally be identified with \[0 \to  \Extz{\Extz{\pi_n(X)}} \to \pi_n(X) \to \Homz{\Homz{\pi_n(X)}} \to 0\] Thus it suffices to show that the cohomology groups of 
\[B^n  \Extz{\Extz{\pi_n(X)}} \qquad \text{and} \qquad B^n  \Homz{\Homz{\pi_n(X)}}\] are finitely generated $\FI$-$\Z$-modules. By assumption, $\Homz{\pi_n(X)}$ and $\Extz{\pi_n(X)}$ are finitely generated $\FI$-$\Z$-modules. By Proposition \ref{EM}, the cohomology of both $B^n \Extz{\Extz{\pi_n(X)}}$ and  $B^n \Homz{\Homz{\pi_n(X)}}$ consists of finitely generated $\FI$-$\Z$-modules.\end{proof}

The following is exactly the integral part of Theorem \ref{thmequiv}.

\begin{theorem}\label{thmequivint}
Let $X$ be a $1$-connected semistrict co-$\FI$-space. Then $H^i(X)$ is a finitely generated $\FI$-$\Z$-module for all $i$ if and only if $\Homz{\pi_i(X)}$ and $\Extz{\pi_i(X)}$ are finitely generated $\FI$-$\Z$-modules for all $i$. The same is true for co-$\FI\#$-spaces and $\FI\#$-$\bZ$-modules.
\end{theorem}

\begin{proof}We again only give the proof for co-$\FI$-spaces and $\FI$-$\bZ$-modules. First assume that $H^i(X)$ is finitely generated for all $i$. By Lemma \ref{whitehead}, $H^i(W_n(X))$ is finitely generated for all $i \geq 0$ and $n \geq 1$. By the universal coefficient theorem and the Hurewicz isomorphism, $H^i(W_i(X)) \cong \Homz{\pi_i(X)}$ and so $\Homz{\pi_i(X)}$ is finitely generated for all $i$. Again by the universal coefficient theorem and the Hurewicz isomorphism, $\Extz{\pi_i(X)}$ injects into $H^{i+1}(W_i(X))$. By Proposition \ref{propfgprops}, $\Extz{\pi_i(X)}$ is finitely generated. 

Conversely assume that $\Homz{\pi_i(X)}$ and $\Extz{\pi_i(X)}$ are finitely generated for all $i$. For $n \geq i+1$, the map $X \to P_n(X)$ induces an isomorphism $H^i(X) \to H^i(P_n(X))$. By Lemma \ref{lempostnikov},  $H^i(P_n(X))$ is finitely generated.
\end{proof}

\subsection{Rational portion of the main theorem}
\label{secequiv}
The previous proof also works in the rational case. However, to get better ranges, we instead opt to work with the Eilenberg-Moore spectral sequence, the Milnor-Moore theorem and Quillen's rational homotopy theory.

\begin{digression}[Naturality of the Eilenberg-Moore spectral sequence] Let $\pi: E \to B$ be a Serre fibration with $B$ a $1$-connected space. The Eilenberg-Moore spectral gives us for a pullback square
\[\xymatrix{E_\pi \ar[r] \ar[d] & E \ar[d]^\pi \\
X \ar[r] & B}\]
a spectral sequence 
\[E_2^{p,q} = \mr{Tor}^{p,q}_{H^*(B)}(H^*(E),H^*(X)) \Rightarrow H^{p+q}(E_\pi)\]
which is natural in maps of such commuting diagrams. Here cohomology is graded negatively, so that this is a second-quadrant spectral sequence. We will be interested in the case
\[\xymatrix{\Omega_* X \ar[r] \ar[d] & P_* X \ar[d]^{\mr{ev}_1} \\
\ast \ar[r] & X}\]
where $\Omega_*X$ are the loop based at $*$ and $P_* X$ are the paths starting at $*$. In particular, we want to apply this to a $1$-connected semistrict co-$\FI$-space $X$ with an arbitrary base point $*_k \in X_k$.
\end{digression}

It is not hard to see that the same techniques as for the Serre spectral sequence allow one to make the Eilenberg-Moore spectral sequences into a spectral sequence of $\FI$-$\bQ$-modules, where the $\FI$-$\bQ$-module structure is induced by that on $H^*(X;\bQ)$. This is the main tool in the proof of the following proposition.

\begin{proposition}\label{proplooprat} Let $X$ be a $1$-connected semistrict co-$\FI$-space with $H^i(X;\Q)$ having weight $\leq c_1 i$ and stability degree $\leq c_2i$. Then $\Omega X$ can again be given the structure of semistrict co-$\FI$-space and $H^i(\Omega X;\bQ)$ will have weight $\leq 2c_1i$ and stability degree $\leq 2(c_1+c_2)i$. If we suppose that $c_1 = c = c_2$, this can be improved to weight $\leq 2ci$ and stability degree $\leq 2ci$. Similarly, if $X$ is a $1$-connected semistrict co-$\FI\#$-space with $H^i(X;\bQ)$ with weight $\leq c i$, then $H^i(\Omega X;\bQ)$ will have weight $\leq 2c i$.\end{proposition}

\begin{proof}As remarked in the previous discussion, the Eilenberg-Moore spectral sequence is a second quadrant spectral sequence of $\FI$-$\bQ$-modules given by
\[E_2^{p,q} = \mr{Tor}^{p,q}_{H^*(X;\bQ)}(\bQ,\bQ) \Rightarrow H^{p+q}(\Omega X;\bQ)\]
where the $\FI$-$\bQ$-module structure is induced by that of $H^*(X;\bQ)$. 

The reduced bar construction gives a choice of $E_1$-page with $E_1^{0,0} = \Q$, $E_1^{0,q}=0$ for $q>0$ and $E_1^{-p,q}$ equal to the degree $q$ part of $\tilde{H}^*(X,\bQ)^{\otimes p}$ if $p$ is positive. This has a differential $E_1^{-p,q} \to E_1^{-p+1,q}$ given by an alternating sum of multiplications. Note that this is $E_1$-page respects the $\FI$-module structure. We note that by Proposition \ref{proptensorrat}, $E_1^{-p,q}$ has weight $\leq c_1 q$ and stability degree $\leq (c_1+c_2)q$. 

Since $X$ is 1-connected, $E_{r}^{-p,q} \cong 0$ unless $q \geq 2p$. Similarly $E_{r}^{p,q} \cong 0$ for $p \geq 0$ unless $(p,q) = (0,0)$. This implies that on the lines of constant $-p+q$, we have that $E_{r}^{-p,q}$ stabilizes when $r = -p+q$. Note that taking homology does not increase weight. Since $E_1^{-p,q}$ has weight $\leq c_1 q$, $E_r^{-p,q}$ also has weight $\leq c_1 q$ for all $r \geq 1$. Thus the worst weight on lines of constant $-p+q$ occurs at $E_{r}^{-(-p+q),2(-p+q)}$ and is given by $\leq 2c_1(-p+q)$.

By Proposition \ref{proptensorrat}, the stability degree of $E_{r+1}^{-p,q}$  will be bounded by the maximum of the stability degrees of $E_r^{-p+r,q-r+1}$, $E_{r}^{-p,q}$ and $E_r^{-p-r,q+r-1}$. We will prove that on lines of constant $-p+q$, the worst stability degree always occurs at $E_{r}^{-(-p+q),2(-p+q)}$. This worst stability degree is given by $\leq 2(c_1+c_2)(-p+q)$ for all $r$, since there are no differentials going in and all differentials go to entries with better stability degree. To see this, note that $E_{r}^{-p,q}$ is certainly bounded by the maximum of the stability degrees of $E_{1}^{-p',q'}$ with $(-p',q')=(-p,q)-\sum_n (r_n,1-r_n)$ over all finite length sequences $\{r_n\}$ with $r_n \geq 1$. Since $(1-r_n)/(r_n) \geq -1$, if $E_{1}^{-p',q'}$ is not zero we must have $q' \leq 2(-p+q)$.

If $-p+q=i$, these results show that $E_{\infty}^{-p,q}$ has weight $\leq 2c_1 i$ and stability degree $\leq 2(c_1+c_2)i$. To get from a result about the $E_{\infty}$-page to a result about $H^i(\Omega X;\bQ)$, we use Proposition \ref{propextrat}. 

If $c_1 = c = c_2$, we can drop these ranges to $\leq 2c i$. In fact, it is $2c\max(2,i)$ a priori, but in the cases $i=0$ or $1$, the cohomology groups are $0$ by the $1$-connectedness assumption.\end{proof}

\begin{digression}[The naturality of the Milnor-Moore theorem] The Milnor-Moore theorem \cite[page 263]{milnormoore} says in particular that if $X$ is a finite type $1$-connected space with base point $*$, then 
\[U(\pi_{*+1}(X,*) \otimes \bQ) \to H_*(\Omega X;\bQ)\]
is an isomorphism of Hopf algebras, where $U$ is the universal enveloping algebra and the Whitehead bracket makes $\pi_{*+1}(X,*) \otimes \bQ$ into a Lie algebra. This implies there are isomorphisms $\pi_{*+1}(X,*) \otimes \bQ \cong \mr{Prim}( H_*(\Omega X;\bQ))$ and $\mr{Hom}_\bZ(\pi_{*+1}(X,*),\bQ) \cong \mr{Ind}( H^*(\Omega X;\bQ))$, where $\mr{Prim}$ denotes primitives and $\mr{Ind}$ indecomposables. All these isomorphism are natural in pointed maps, but by the same trick as for the Serre and Eilenberg-Moore spectral sequence, can be made natural for maps between $1$-connected spaces that do not preserve the base point.
\end{digression}

\begin{digression}[The naturality of the Quillen spectral sequence] In \cite{quillenrat}, Quillen described a functor $\lambda$ from pointed finite type $1$-connected spaces to the category of differential-graded Lie algebras over $\bQ$. This has the property that the $i$th homology of $\lambda(X)$ is naturally isomorphic to $\pi_{i+1}(X) \otimes \bQ$, and the $i$th dg-Lie algebra cohomology $H^i_{CE}(\lambda(X))$ of $\lambda(X)$ is naturally isomorphic to $H^i(X;\bQ)$. The dg-Lie algebra cohomology can be defined using the Chevalley-Eilenberg complex. Filtering this complex by the homological degree of $\lambda(X)$, we obtain the Quillen spectral sequence
\[E^{p,q}_2 = H^q_{CE}(\pi_{p+1}(X) \otimes \bQ) \Rightarrow H^{p+q}(X;\bQ)\]
where the Lie bracket on $\pi_{*+1}(X) \otimes \bQ$ is given by the Whitehead bracket. The groups $H^*_{CE}(\pi_{*+1}(X) \otimes \bQ)$ are naturally isomorphic to the homology a chain complex whose underlying graded vector spaces are given by the free graded-commutative algebra on $\mr{Hom}_\bZ(\pi_{*}(X),\bQ)$, with differential coming from the Lie bracket. All of this is natural in pointed maps. Like the Serre and Eilenberg-Moore spectral sequences and the Milnor-Moore theorem, the Quillen spectral sequence can be made natural for maps between $1$-connected spaces that do not preserve the base point.
\end{digression}

The following is the rational part of Theorem \ref{thmequiv}.

\begin{theorem}\label{thmequivrat} Let $X$ be a $1$-connected semistrict co-$\FI$-space.
\begin{enumerate}[(i)]
\item Suppose that $H^i(X;\Q)$ has weight $\leq c_1 i$ and stability degree $\leq c_2i$. Then $\mr{Hom}_{\bZ}(\pi_i(X),\bQ)$ has weight $\leq 2c_1(i-1)$ and stability degree $\leq (4c_1+2c_2)(i-1)$. If $c_1 = c = c_2$ we can improve this to weight $\leq 2c(i-1)$ and stability degree $\leq 4c(i-1)$.
\item Suppose $\mr{Hom}_\bZ(\pi_i(X),\bQ)$ has weight $\leq c_1i$ and stability degree $\leq c_2i$. Then $H^i(X;\bQ)$ has weight $\leq c_1i$ and stability degree $\leq (c_1+c_2)(2i+1)$. If $c_1 = c = c_2$, this can be improved to weight $\leq ci$ and stability degree $\leq c(2i+1)$.
\end{enumerate}
The same is true for $\FI\#$-$\bQ$-modules, and in that case only weight is relevant.\end{theorem}

\begin{proof}The Milnor-Moore theorem says that $\mr{Hom}_{\bZ}(\pi_{i+1}(X),\bQ)$ is given by the degree $i$ indecomposables of $A:= H^i(\Omega X;\bQ)$, i.e. the degree $i$ part of the cokernel of the map $A_+^{\otimes 2} \to A$. We computed the weight and stability degree of $A$ in Proposition \ref{proplooprat}. It now suffices to remark that the degree $i$ part of $A_+^{\otimes 2}$ has weight $\leq 2c_1i$ and stability degree $\leq (4c_1+2c_2)i$ by Proposition \ref{proptensorrat}. If $c_1 = c = c_2$ this drops to $\leq 2ci$ and $\leq 4ci$ respectively.

Quillen's approach to rational homotopy theory gives us a spectral sequence computing $H^*(X;\bQ)$ from the graded-commutative algebra on $\mr{Hom}_\bZ(\pi_{*}(X),\bQ)$. In degree $i$ the free graded commutative algebra then has weight $\leq c_1 i$ and stability degree $\leq (c_1+c_2)i$. A similar argument about the stability ranges of various entries of the spectral sequence as the one in Proposition \ref{proplooprat} gives the desired result. More precisely, we first lose some range when computing the cohomology of the Chevalley-Eilenberg complex; it goes down from $\leq (c_1+c_2)i$ to $\leq (c_1+c_2)(i+1)$. Secondly the possibility of differentials in the spectral sequence makes it go down to $\leq (c_1+c_2)(2i+1)$.
\end{proof}

\begin{remark}The previous theorem takes linear ranges as input, but has affine ranges as output. However, since by assumption $H^0$ is isomorphic to $\bQ$ and $H^1$ is $0$, an affine range of $\leq ci+a$ implies a linear range of $\leq (c+\frac{a}{2})i$ if $a>0$ and $\leq ci$ if $a \leq 0$. This will allow us to iterate the theorem.\end{remark}

\section{Applications}\label{secap}
In this section we discuss applications of Theorem \ref{thmequiv}, both in topology and algebra.

\subsection{Configuration spaces} 
\label{subconfig}

Let $F(M)$ be the co-$\FI$-space whose value on a set is the space of embeddings of that set into $M$. When $M$ is non-compact, in addition to deleting points, one can also bring points in from infinity. This allows one to define a semistrict co-$\FI\#$-space structure on $F(M)$ (see Section 6.4 of \cite{CEF}). The following combines Application 2 of \cite{CEFN} and Theorem 6.3.1 of \cite{CEF}. 

\begin{theorem}For $M$ a connected orientable manifold of dimension at least $2$, $H^i(F(M))$ is finitely generated and the weight and stability degree of $H^i(F(M);\Q)$ are $\leq i$.
\label{homconf}
\end{theorem}

We now restrict our attention to the case that $M$ is $1$-connected (hence orientable) and of dimension at least $3$. In this case, the fact that codimension of the fat diagonal is at least 3 implies that $F_k(M)$ is $1$-connected for all $k$. Combining Theorem \ref{homconf} with Theorem \ref{thmequiv} gives the following. 

\begin{corollary}Let $M$ be a $1$-connected manifold of dimension at least $3$.
\begin{enumerate}[(i)]
\item The $\FI$-$\Z$-modules $\Homz{\pi_i(F(M))}$ and $\Extz{\pi_i(F(M))}$ are finitely generated.
\item  The $\FI$-$\Q$-module $\mr{Hom}_\Z(\pi_i(F(M)), \bQ)$ has weight $\leq 2(i-1)$ and stability degree $\leq 4(i-1)$. If $M$ is non-compact, the $\FI\#$-$\Q$-module $\mr{Hom}_\Z(\pi_i(F(M)), \bQ)$ has weight $\leq 2(i-1)$.
\end{enumerate}
\end{corollary}

Combining this with Lemma \ref{FIrepstab} and the isomorphism $\pi_i(C_k(M)) \cong \pi_i(F_k(M))$ for $i \geq 2$ proves Theorems \ref{config} and \ref{configint}. 

\begin{example}
We now return to the example of $\pi_{n-1}(C_k(\R^n))$ considered in the introduction. Since $\pi_{n-1}(C_k(\R^n)) \otimes \Q \cong \Q^{k \choose 2}$ for $n \geq 3$, these groups do not stabilize in the traditional sense. However, using the calculations of Cohen in III.4 and III.5 of \cite{CLM}, the techniques of Church and Farb in Chapter 4 of \cite{CF} and the Hurewicz isomorphism, one can compute the representations that appear and observe representation stability. For $k \geq 4$ and $n \geq 3$, we have: $$\pi_{n-1}(C_k(\R^n)) \otimes \Q \cong  \begin{cases}
 V(0)_k \oplus V(1)_k \oplus V(2)_k  & \text{if n even} \\
V(1)_k \oplus V(1 \geq 1)_k & \text{if n odd} \end{cases} $$

\end{example}

\begin{remark}\label{remdim2} For completeness we also note that the higher homotopy groups of configuration spaces of points in $1$-connected manifolds of dimension $1$ and $2$ also stabilize. In dimension $1$, $\R$ is the only $1$-connected manifold and $C_k(\R)$ is contractible. In dimension $2$, there are two $1$-connected manifolds, $\R^2$ and $S^2$. It is well known that $C_k(\R^2)$ is aspherical \cite{FN}. From \cite{FV}, it follows that, for $i \geq 2$, we have $\pi_i(C_k(S^2)) \cong \pi_i(S^2)$ if $k =1$ or $2$ and $\pi_i(C_k(S^2)) \cong \pi_i(S^3)$ for $k \geq 3$. These explicit calculations imply that in low dimensions the higher homotopy groups stabilize in the traditional sense.

%For manifolds of dimension $1$ and $2$ the following is the case for the higher homotopy groups. In the $1$-dimensional case, for all $k \geq 1$ we have that $C_k(\bR)$ is contractible and $C_k(S^1)$ is homotopy equivalent to $S^1$. 

%In the $2$-dimensional case of a connected surface $\Sigma$, one can use the fiber sequences
%\[\Sigma \backslash \{k\text{ pts}\} \to F_{k+1}(\Sigma) \to F_k(\Sigma)\]
%and remark that $\Sigma \backslash \{k\text{ pts}\}$ is homotopy equivalent to a wedge of circles for $k \geq 1$. This implies that $\pi_i(F_{k}(\Sigma)) \cong \pi_i(\Sigma)$ is an isomorphism for $i \geq 3$ and $\pi_2(F_{k+1}(\Sigma)) \to \pi_2(F_k(\Sigma))$ is an injection for all $k$. Thus in these case the higher homotopy groups stabilize as abelian groups. Also remark that unless $\Sigma = S^2$, $\pi_i(\Sigma) = 0$ for $i \geq 2$.
\end{remark}

\subsection{Other fundamental groups}

\label{secG}

In this subsection, we sketch an extension of our results to the homotopy groups of configuration space of points in manifolds with certain non-trivial fundamental groups. In this context, the relevant category is not $\FI$ but $\FI(G)$. 

\begin{definition}
Let $\FI(G)$ denote the category whose objects are free $G$-sets with finitely many orbits and whose morphisms are injective G-equivariant maps.
\end{definition}

From now let $G=\pi_1(M)$ for some fixed manifold $M$ of dimension $\geq 3$. The category $\FI(G)$ is relevant because the groups $\pi_i(C_k(M))$ for $i \geq 2$ naturally assemble to form a co-$\FI(G)$-module. To see this, we will describe a co-$\FI(G)$-space structure on the universal covers of $C_k(M)$, denote by $\smash{\widetilde{C_k(M)}}$. Let $F^{G}(M)$ be the co-$\FI(G)$-space whose value on a $G$-set $X$ is $\mr{Emb}^G(X,\tilde M)$, the space of $G$-equivariant embeddings of $X$ into $\tilde M$, the universal cover of $M$. 

\begin{proposition}

If the dimension of $M$ is at least 3, then the value of $F^{G}(M)$ on the $G^k$ is homeomorphic to $\smash{\widetilde{C_k(M)}}$.

\end{proposition}

\begin{proof}

The fiber of the natural map $F^G(M)(G^k) \to F_k(M)$ is $G^k$. When $M$ is  at least $3$ dimensional, $G^k \cong \pi_1(F_k(M))$ and $F^G(M)(G^k)$ is $0$-connected. From the long exact sequence of homotopy groups we conclude that $F^{G}(M)(G^k)$ is the universal cover of $F_k(M)$. Since $F_k(M)$ covers $C_k(M)$ and is $0$-connected, we conclude that $F^{G}(M)(G^k)$ is also the universal cover of $C_k(M)$. 
\end{proof}

Since $F^G(M)$ is $1$-connected, its derived dual homotopy groups are naturally $\FI(G)$-modules and these homotopy groups agree with those of unordered configuration spaces. Recall that a group is called \emph{polycyclic} if it has a finite length filtration by normal subgroups whose filtration quotients are cyclic and a group is called \emph{polycyclic-by-finite} if it is has a finite index polycyclic subgroup. In particular, all finite groups as well as $\Z^k$ are polycyclic-by-finite. In \cite{SS2}, Sam and Snowden proved the following analogue of part (iii) of \ref{propfgprops}.

\begin{theorem}
Let $G$ be polycyclic-by-finite and $A$ be a sub-$\FI(G)$-$\Z$-module of $B$. If $B$ is finitely generated, so is $A$.
\label{FIGnoeth}
\end{theorem}

By copying the proof of Theorem E of \cite{CEFN}, replacing $F(M)$ with $F^G(M)$, and invoking the Noetherian property of $\FI(G)$, i.e. Theorem \ref{FIGnoeth}, one gets the following proposition.

\begin{proposition}
\label{homfgG}
Let $M$ be a connected manifold of dimension at least $3$ with polycyclic-by-finite fundamental group $G$. Then for all $i$, $H^i(F^G(M);\Z)$ is finitely generated as an $\FI(G)$-module.
\end{proposition}

By \cite{SS2},  $\FI(G)$ satisfies all conditions in Remark \ref{remintegralreqs} when $G$ is polycyclic-by-finite. Thus one can generalize Theorem \ref{thmequivint} as follows.

\begin{proposition}\label{thmequivG}
Let $G$ be polycyclic-by-finite and $X$ be a $1$-connected semistrict co-$\FI(G)$-space. Then $H^i(X)$ is a finitely generated $\FI(G)$-$\Z$-module for all $i$ if and only if $\Homz{\pi_i(X)}$ and $\Extz{\pi_i(X)}$ are finitely generated $\FI(G)$-$\Z$-modules for all $i$.
\end{proposition}

Combining Propositions \ref{homfgG} and \ref{thmequivG}, we get the following. 

\begin{theorem} \label{stabG}
Let $M$ be a connected manifold of dimension at least $3$ with polycyclic-by-finite fundamental group $G$. For $i \geq 2$, $\Homz{\pi_i(C_k(M))}$ and $\Extz{\pi_i(C_k(M))}$ are finitely generated as $\FI(G)$-$\bZ$-modules.
\end{theorem}

When working with rational coefficients and $G$ finite, we believe that finite generation is equivalent to some version of representation stability and that a quantitative version of Theorem \ref{stabG} holds.

\subsection{Other approaches for configuration spaces}\label{subconfother}

We will now discuss related results, which can be adapted to deduce particular cases of the results in the previous two sections.

\subsubsection{Work of Cohen and Gitler} Cohen and Gitler have explicitly determined the rational homotopy groups of certain configuration spaces. Theorem 2.3 of \cite{cohengitler} says that the primitive elements of $H_*(\Omega F_k(\bR^n))$ can be identified with a graded Lie algebra $\mathcal L_k(n-2)$, obtained as the quotient of a direct sum of free Lie algebras by infinitesimal braid relations. Using the Milnor-Moore theorem, this identifies the rational homotopy groups and from their description one can read off finite generation. More recently, Berglund gave an elegant approach to this calculation using Koszul algebras (see Example 32 of  \cite{berglundkoszul}). Cohen and Gitler's results extend to more general classes of manifolds, in particular $1$-connected $p$-manifolds of dimension at least 3 (Theorem 2.4 of \cite{cohengitler}). A $p$-manifold is a manifold obtained by removing a point from some manifold.

\subsubsection{Work of Levitt}In Corollary 1.4 of \cite{levitt}, Levitt shows that if $M$ is (i) non-compact and connected, or (ii) is compact and has Euler characteristic zero, we have an isomorphism \[\pi_i(F_k(M)) \cong \bigoplus_{0 \leq j \leq k-1} \pi_i(M \backslash \{j \text{ pts}\})\]

This manifestly breaks the $\fS_k$-symmetry. However, in special circumstances we can still recover information from it. In particular in case (i), we have $M \backslash \{j \text{ pts}\} \simeq M \vee \bigvee_{j} S^{\dim M-1}$, so one might be able to use this to bound the number of generators of $\pi_i(F_k(M))$ by a polynomial. If $M$ were additionally $1$-connected, such a bound implies finite generation, as the homotopy groups of ordered configuration spaces of non-compact manifolds are $\FI\#$-modules and we can thus apply Theorem 4.1.7 of \cite{CEF}.

\begin{example}One can prove finite generation integrally for non-compact manifolds homotopy equivalent to wedges of spheres of dimension $\geq 2$ using Hilton's theorem \cite{hilton}, which determines the homotopy groups of a wedge of spheres in terms of the unstable homotopy groups of spheres and a basis for the free Lie algebra.\end{example}

\begin{example}One can prove finite generation over $\Q$ for any non-compact $1$-connected $M$, using Quillen's rational homotopy theory \cite{quillenrat}. Quillen's theory says that every pointed $1$-connected space $X$ of finite type has a minimal dg-Lie algebra model $\mathcal L(X)$, whose homology groups are the rational homotopy groups of $\Omega X$. For example, an $n$-sphere has as minimal dg-Lie algebra model the free graded Lie algebra $\mathbb L[e_{n-1}]$ on a generator of degree $n-1$. A dg-Lie algebra model for $X \vee Y$ is given the coproduct $\mathcal L(X)*\mathcal L(Y)$ of the minimal models of $X$ and $Y$. Thus it suffices to note that in fixed degree the dimension of $\mathcal L(M) * \mathbb L[e_{n-1}^1,\ldots,e_{n-1}^k]$ is polynomial in $k$. 

In fact, expanding on this approach allows one to prove Theorem \ref{config} without explicit ranges, by the following reduction from the closed to the non-compact case. There is a shift map $\mr{sh}: \FI \to \FI$ that sends a finite $S$ to $S \sqcup *$. To prove that an $\cat{FI}$-$R$-module $F$ is finitely generated, it suffices to prove that $F \circ \mr{sh}$ is finitely generated. In the case of configuration spaces -- that is, $F = H^i(F_k(M);\bQ)$ -- the $\FI$-module $F \circ \mr{sh}$ in degree $k$ consists of the $i$th rational cohomology group of $F_{k+1}(M)$, but we consider one point as being labeled by the special element $*$. Remembering just that special point gives a fibration 
\[F_k(M \backslash \{\mr{pt}\}) \to F_{k+1}(M) \to M\]
This induces a long exact sequences of $\FI$-modules on dual rational homotopy groups, from which one computes the weight and stability degree of the $\FI$-modules $\mr{Hom}_\bZ(\pi_i(F_k(M)),\bQ)$ composed with $\mr{sh}$.\end{example}

\subsection{Subgroups of automorphism groups of manifolds} Let $G$ be a topological group that acts on $M$. Writing $F_k(M) = \mr{Emb}([k],M)$ one sees that the actions of $\fS_k$ and $G$ on $F_k(M)$ commute. Using this one concludes that
\[F_k(M) \to F_k(M) \times_G EG \to BG\]
is a fiber sequence of $\fS_k$-spaces, if the action of $\fS_k$ on $BG$ is set to be trivial. These fiber sequences assemble into a fiber sequence of co-$\FI$-spaces.

If we pick arbitrary base points $*_k$ in $F_k(M)$ and a base point $* \in EG$, we get a base point $*_k$ in $F_k(M) \times_G EG$ as the image of $(*_k,*)$. The base point in $EG$ gives a base point $*$ in $BG$ as well. Even though $F_k(M) \times_G EG$ might not be $1$-connected, the groups $\pi_i(F_k(M) \times_G EG,*_k)$ for $i \geq 1$ can still be given the structure of a co-$\FI$-$\Z$-module without any additional choices. To see this, note that for each injection $\tau: [k] \to [k']$ the map $\tau^*: F_k(M) \times_G EG \to F_k(M) \times_G EG$ preserve the fibers and these are $1$-connected.

\begin{corollary}Let $M$ be a $1$-connected manifold of dimension at least $3$.
\begin{enumerate}[(i)]\item If $\pi_i(BG,*)$ is a finitely generated abelian group for $1 \leq i \leq j$, then for $1 \leq i \leq j-1$ the groups $\Homz{\pi_i(F_k(M) \times_G EG,*_k)}$ and $\Extz{\pi_i(F_k(M) \times_G EG,*_k)}$ assemble to finitely generated $\FI$-$\bZ$-modules.
\item If $\pi_i(BG,*) \otimes \bQ$ is finite dimensional for $1 \leq i \leq j$, then for $1 \leq i \leq j-1$ the vector spaces $\Homq{\pi_i(F_k(M) \times_G EG,*_k)}$ assemble to $\FI$-$\bQ$-modules with weight $\leq 2(i-1)$ and stability degree $\leq 4(i-1)$.
\end{enumerate}
\label{groups}
\end{corollary}

\begin{proof}We prove the integral case, leaving the rational case to the reader. Consider again the fiber sequence of co-$\FI$-spaces
\[F_k(M) \to F_k(M) \times_G EG \to BG\]
The long exact sequences of homotopy groups combine into a long exact sequence of co-$\FI$-$\Z$ modules, which we think of as a chain complex $C_*$. Recall that $\Homz{-}$ and $\Extz{-}$ are the derived functors of $F:= \Homz{-}$. We can try to compute the hyperhomology $\mathbb H_*(F,C_*)$ of $F$ applied to $C_*$; this is by definition the result of resolving $C_*$ by injectives $I_{**}$, applying $F$ and taking total cohomology. We remark that a presheaf category with values in $\cat{Ab}$ has enough injectives, and these are furthermore objectwise injectives. Filtering $I_{**}$ in either direction gives rise to two hyperhomology spectral sequences
\begin{align*}E^2_{p,q} &= R^p F(H_q(C_*)) \Rightarrow \mathbb H_{p+q}(F,C_*) \\
E^1_{p,q} &= R^p F(C_q) \Rightarrow \mathbb H_{p+q}(F,C_*)\end{align*}

Since $C_*$ is exact, the first one tells use that the hyperhomology vanishes. This means that the second one converges to zero. Since $F = \Homz{-}$ only has non-zero derived functors $R^0 F= \Homz{-}$ and $R^1 F= \Extz{-}$, all differentials going into and out of $\Homz{\pi_*(F_k(M) \times_G EG)}$ and $\Extz{\pi_*(F_k(M) \times_G EG)}$ come from or go into the groups $\Homz{\pi_*(F_k(M))}$, $\Extz{\pi_*(F_k(M))}$, $\Homz{\pi_*(BG)}$ or $\Extz{\pi_*(BG)}$, all of which are finitely generated. This implies that $\Homz{\pi_*(F_k(M) \times_G EG)}$ and $\Extz{\pi_*(F_k(M) \times_G EG)}$ are finitely generated using that finite generation is preserved by taking subgroups, quotients and extensions.

Alternatively, an elementary proof can be given by splitting the long exact sequence into short exact sequences and remarking that every short exact sequence gives a six-term long exact sequence involving $\Homz{-}$ and $\Extz{-}$.\end{proof}

We fix an infinite sequence of distinct points $m_1,m_2, \ldots$ in $M$ and let $G_k$ denote the closed subgroup of $G$ fixing $m_1,\ldots,m_k$. Important examples come from those $G$ with the following two properties:
\begin{enumerate}[(i)]
\item $G$ acts $k$-transitively on $M$ for all $k \geq 1$, which means that the induced action on $F_k(M)$ is transitive for all $k \geq 1$. This implies that $G_k$ does not depend on the points $m_1,\ldots,m_k$ up to homeomorphism.
\item There are local sections to the quotient map $G \to G/G_k$. This implies that $F_k(M) \cong G/G_k$ and $F_k(M) \times_G EG \simeq BG_k$.
\end{enumerate}

Two examples of groups $G$ having these properties are given by $\mr{Diff}(M)$ and $\mr{Homeo}(M)$, defined as follows. If $M$ is a compact smooth manifold, $\Diff(M)$ is the group of diffeomorphisms of $M$ with the $C^\infty$ topology. If $M$ is a compact topological manifold, $\Homeo(M)$ is the group of homeomorphisms of $M$ with the compact-open topology. One could also consider PL-homeomorphisms of a PL-manifold by working with simplicial groups.

\begin{corollary}Let $M$ be a $1$-connected manifold of dimension at least $3$ and let $\mr{Aut}$ denote either $\Diff$ or $\Homeo$.
\begin{enumerate}[(i)]\item If $\pi_i(B\mr{Aut}(M),*)$ is a finitely generated abelian group for $1 \leq i \leq j$, then both sequences $\Homz{\pi_i(\mr{Aut}(M)_k,*_k)}$ and $\Extz{\pi_i(\mr{Aut}(M)_k,*_k)}$ are finitely generated $\FI$-$\bZ$-modules for $1 \leq i \leq j-1$.
\item If $\pi_i(B\mr{Aut}(M),*) \otimes \bQ$ is finite dimensional for $1 \leq i \leq j$, then $\Homq{\pi_i(\mr{Aut}(M)_k,*_k)}$ are $\FI$-$\Q$-modules of weight $\leq 2(i-1)$ and stability degree $\leq 4(i-1)$ for $1 \leq i \leq j-1$.
\end{enumerate}
\label{diff}
\end{corollary}

\begin{remark}The diffeomorphism and homeomorphism groups of $1$-connected manifolds of high dimension are known to have finite dimensional rational homotopy groups in a range. This was proven for spheres in \cite{FHs} and for general manifolds in \cite{burghelea}. Also we remark that Hatcher's proof of Smale's conjecture \cite{SmaleConjecture} implies that $\pi_i(\Diff(S^3),\mr{id})$ is finitely generated for all $i$.
\end{remark}

\begin{remark}Since $\pi_{i}(B\mr{Aut}(M)_k,*) \cong \pi_{i+1}(\mr{Aut}(M)_k,\mr{id})$, Corollary \ref{diff} applies equally well to $\mr{Aut}(M)_k$. Let $\mr{Aut}(M)_{(n)}$ be the subgroup of $\mr{Aut}(M)$ fixing $\{m_1, \ldots, m_n\}$ set-wise but not necessarily point-wise. Since $B\mr{Aut}(M)_n \to B\mr{Aut}(M)_{(n)}$ is a cover, the higher homotopy groups of $B\mr{Aut}(M)_{(n)}$ also exhibit representation stability. \end{remark}

\begin{remark}In \cite{JR}, Jimenez Rolland proved that $H^*(B\Diff(M)_k;\bQ)$ has representation stability under the condition that $H^*(B\Diff(M);\bQ)$ is finite dimensional in each degree. We discussed above that the spaces $B\Diff(M)_k$ assemble to a co-$\FI$-space. Therefore, Corollary \ref{diff} would follow from her work and Theorem \ref{thmequiv} in the case that each $B\Diff(M)_k$ is $1$-connected. We do not know of any example where this is the case, though there are examples for manifolds with boundary, e.g. $B\mr{Diff}(D^3;\partial D^3)$.\end{remark}

\subsection{Products and bounded products}
In Section \ref{subconfig} we used representation stability for  cohomology groups to deduce representation stability for dual  homotopy groups. In this section, we reverse this, giving an example where it is easier to show dual homotopy groups stabilize than to show that the cohomology groups do. 

Given a space $Z$, we consider the co-$\FI\#$-space whose value on a finite set $S$ is $Z^S$. Representation stability for $\{H^*(Z^k;\bQ)\}_{k \geq 0}$ follows from Proposition 6.1.2 of \cite{CEF}. Thus the following is not a new result, but we still give a proof as an example of the dual-homotopy-to-cohomology direction of Theorem \ref{thmequiv}.

\begin{corollary}
Let $Z$ be a $1$-connected space of finite type.
\begin{enumerate}[(i)]
\item For $i \geq 0$, $\FI\#$-$\bZ$-modules $H^i(Z^k)$ are finitely generated. 
\item For $i \geq 0$, $\FI\#$-$\bQ$-modules $H^i(Z^k;\bQ)$ have weight $\leq i$.
\end{enumerate}
\end{corollary}

\begin{proof}
Since $\pi_i(Z^k)=\oplus_k \pi_i(Z)$, $\pi_i(Z^k)$ is a direct sum of $M(1)$'s. Integrally these are finitely generated and rationally these have weight $\leq 1$, hence their duals have the same property. Now apply the homotopy-to-cohomology part of Theorem \ref{thmequiv}.
\end{proof}

\begin{remark}It follows that the sequence $H^*(Z^n;\bQ)$ has representation stability. Since $H^*(Z^n;\bQ)^{\fS_n} \cong H^*(Z^n/\fS_n;\bQ)$, this gives a new proof of rational homological stability for symmetric products \cite{Srod}. Using results of Kallel and Saihi, this implies a similar result in a range for bounded products. Fix an integer $c \geq 1$ and let $\Delta^c(Z,k)$ be the subspace of $Z^k$ of sequences of points such that no $c$-tuple coincides. For example, $\Delta^1(Z,k) = F_k(Z)$ and the quotient $\Delta^c(Z,k)/\fS_k$ is also known as the bounded symmetric product $\mr{Sym}^{\leq c}_k(Z)$. 

Let $Z$ be a space such that for any $z \in Z$ and open $U$ containing $z$ there is an open $V$ containing $z$ such that $V \backslash \{z\}$ is $r$-connected. In \cite{kalleldiagonal}, Kallel and Saihi show that the inclusion $\Delta^c(Z,k) \to Z^k$ is $(r+1)$-connected, so in that range our results for products extend to bounded products. In unpublished work, they improve this range to $rc + 2c-2$.\end{remark}

\subsection{Loop spaces of suspensions} 

One can also use the cohomology-to-dual-homotopy and dual-homotopy-to-cohomology in tandem. For example, Theorem \ref{thmequiv} has the following corollary for the cohomology and dual homotopy groups of loop spaces of suspensions of co-$\FI$-space.

\begin{corollary}\label{corloopsusp}
Let $X$ be an $(m-n)$-connected semistrict co-$\FI$-space of finite type.
\begin{enumerate}[(i)]
\item Suppose that $H^i(X)$ is a finitely generated $\FI$-$\bZ$-module for all $i$. Then the $\FI$-$\bZ$-modules $H^i(\Omega^m \Sigma^n X)$, $\Extz{\pi_i(\Omega^m \Sigma^n X)}$ and $\Homz{\pi_i(\Omega^m \Sigma^n X)}$ are finitely generated for all $i$. 
\item If $H^i(X;\Q)$ has weight $\leq c_1i$ and stability degree $\leq c_2i$ then $H^i(\Omega^m \Sigma^n X;\Q)$ has weight $\leq (m+2)c_1i$ and stability degree $\leq ((2m+3)c_1+(m+1)c_2)(2i+1)$, and $\Homq{\pi_i(\Omega^m \Sigma^n X)}$ has weight $\leq (m+2)c_1i$ and stability degree $\leq (m+1)(c_1+c_2)i$.
\end{enumerate} 
\end{corollary}

\begin{proof}The proof is a repeated application of Theorem \ref{thmequiv} and the fact that $\Sigma$ shifts homology groups and $\Omega$ shifts homotopy groups. In the rational case, we note that $H^i(\Sigma^n X;\bQ)$ still has weight $\leq c_1i$ and stability degree $\leq c_2i$, though one can give better affine ranges. From this we conclude that the dual homotopy groups have weight $\leq c_1i$ and stability degree $\leq 2(c_1+c_2)(i-1)$. Looping $m$ times proves that the dual homotopy groups of $\Omega^m \Sigma^n X$ have weight $\leq 2c_1(i+m)$ and stability degree $\leq 2(c_1+c_2)(i+m-1)$. This gives linear ranges as follows: weight $\leq (m+2)c_1i$ and stability degree $\leq (m+1)(c_1+c_2)i$. Finally, going back to cohomology gives weight $\leq 2(m+2)c_1i$ and stability degree $\leq ((2m+3)c_1+(m+1)c_2)(2i+1)$. \end{proof}

\subsection{Algebraic examples}
\label{algebra}
Finally, we give some algebraic consequences. Any contravariant functor from the category of finite sets with all set maps to the category of $R$-modules is also an $\FI\#$-$R$-module. The first example we consider is $\mathcal G^{d,m}$, the functor which sends a set to dual of the free graded rational Gerstenhaber algebra on the set. Here the $d$ is the degree of the generators and $m-1$ is the degree of the Gerstenhaber bracket. Note that for $m=1$, this is isomorphic as a graded vector space to the free associative algebra. Let $\mathcal L^{d,m}$ as the free graded rational Lie algebra functor on generators of degree $d$ and bracket of degree $m-1$. Let $\mathcal G^{d,m,l}$ and $\mathcal L^{d,m,l}$ denote their $l$th graded pieces. Note that representation stability for free Lie algebras was originally proven in Section 5.3 of \cite{CF}, but without an explicit stability range. 

\begin{corollary}Take $d \geq 2$. 
\begin{enumerate}[(i)]
\item The $\FI\#$-$\bQ$-module $\mathcal G^{d,m,l}$ has weight $\leq \frac{m+2}{m+d}l$.
\item The $\FI\#$-$\bQ$-module $\mathcal L^{d,m,l}$ has weight $\leq \frac{m+2}{m+d}l$
\end{enumerate} 
\end{corollary}

\begin{proof}
Let $\mathcal S^{d+m}$ be the $(d+m)$-fold suspension functor viewed as co-$\FI\#$-space. That is $\mathcal S^{d+m}([n]) =  \bigvee_{i=1}^n S^{d+m}$. This co-$\FI\#$-space is relevant because $H^l(\Omega^m S^{d+m}([n]);\bQ)$ is the dual of $\mathcal G^{d,m,l}_n$ and $\Homq{\pi_l(\Omega^m S^{d+m}[n])}$ is the dual of $\mathcal L^{d,m,l}_n$. The restriction $d \geq 2$ is guarantee that $\Omega^m S^{d+m}[n]$ is $1$-connected.

Note that $H^i(\mathcal S^{d+m}_n) = 0$ for $i \neq 0$ or $d+m$. As representations of $\fS_n$, $H^0(\mathcal S^{d+m}_n;\bQ)$ is the trivial representation and $H^m(\mathcal S^{d+m}_n;\bQ)$ is the permutation representation. In terms of $\FI$-modules, we have $H^0(\mathcal S^{d+m};\bQ) \cong M(0) \otimes \bQ$ and $H^{d+m}(\mathcal S^{d+m};\bQ) \cong M(1) \otimes \bQ$. By Proposition 3.2.4 of \cite{CEF}, the weight of $M(k)$ is $k$. Therefore $\mathcal S^{d+m}$ is a co-$\FI\#$ space with $H^i(\mathcal S^{d+m};\bQ)$ of weight $\leq \frac{i}{m+d}$. The corollary now follows from Corollary \ref{corloopsusp} and Proposition \ref{propweightdualrat}.\end{proof}

\begin{remark}
Using similar arguments for homology with $\F_p$-coefficients, one can prove finite generation for the graded pieces of the $\FI$-$\F_p$-modules obtained by sending a finite set $S$ to the unstable variants of the Dyer-Lashof algebra considered in \cite{CLM}, freely generated by $S$ and suitably graded. 
\end{remark}

\bibliographystyle{alpha}
\bibliography{kupersmiller}

\end{document}